\documentclass[12pt]{amsart}
\usepackage[pdftex,colorlinks,citecolor=blue,urlcolor=blue]{hyperref}
\usepackage{amssymb}
\usepackage{graphicx,color}
\usepackage{amsmath}
\usepackage{comment}
\usepackage{bbm}
\usepackage{amssymb}
\usepackage{cleveref}
\usepackage{enumerate}
\usepackage{mathtools}
\usepackage{tikz}
\usepackage{soul}
\usetikzlibrary{decorations.pathreplacing,calc}
\newtheorem{theorem}{Theorem}[section]
\newtheorem{proposition}[theorem]{Proposition}
\newtheorem{lemma}[theorem]{Lemma}
\newtheorem{corollary}[theorem]{Corollary}
\newtheorem{construction}[theorem]{Construction}
\newtheorem{definition}[theorem]{Definition}

\newtheorem{example}[theorem]{Example}

\newtheorem{observation}[theorem]{Observation}

\DeclareMathOperator{\supp}{\mathrm{supp}}

\newcommand{\N}{\mathbb{N}} 
\newcommand{\Z}{\mathbb{Z}}
\newcommand{\C}{\mathcal{C}}
\newcommand{\R}{\mathbb{R}}

\newcommand{\conv}{\mathop{\rm conv}\nolimits}
\newcommand{\Pc}{\mathcal{P}}

\newcommand{\ESP}{\mathrm{ESP}}
\newcommand{\CBP}{\mathrm{CBP}}
\newcommand{\CSP}{\mathrm{CSP}}
\newcommand{\TECSP}{\mathrm{TECSP}}
\renewcommand{\P}{\mathcal{P}}

\newcommand{\K}{\mathcal{K}}
\newcommand{\G}{\mathcal{G}}

\newcommand{\I}{\mathcal{I}}
\newcommand{\T}{\mathcal{T}}
\newcommand{\B}{\mathcal{B}}
\newcommand{\A}{\mathcal{A}}
\renewcommand{\H}{\mathcal{H}}
\newcommand{\F}{\mathcal{F}}
\newcommand{\init}{\mathrm{in}}
\newcommand{\rev}{\mathrm{rev}}

\newcommand{\Xr}[1]{}
\newcommand{\Xn}[1]{{\protect#1}}

\title{On the connected blocks polytope}
\author{Justus Bruckamp, Markus Chimani, Martina Juhnke}
\address{Faculty of Mathematics/Computer Science, University of Osnabr\"{u}ck, Germany}
\email{justus.bruckamp@uni-osnabrueck.de}
\email{markus.chimani@uni-osnabrueck.de}
\email{martina.juhnke@uni-osnabrueck.de}

\begin{document}

\maketitle
\begin{abstract}
In this paper, we study the connected blocks polytope, which, apart from its own merits, can be seen as the generalization of certain connectivity based or Eulerian subgraph polytopes. We provide a complete facet description of this polytope, characterize its edges and show that it is Hirsch. We also show that connected blocks polytopes admit a regular unimodular triangulation by constructing a squarefree Gr\"obner basis. In addition, we prove that the polytope is Gorenstein of index $2$ and that its $h^\ast$-vector is unimodal.
\end{abstract}

\section{Introduction}
We study the \emph{connected blocks polytope} $\CBP(G)$ of a connected graph $G$, which is closely related to several polytopes from literature. Our study generalizes some known results about these polytopes and also considers new aspects that have not been studied so far at all. Our aim is to understand connectivity in a polyhedral way and we are interested in finding sets of blocks of a graph that induce a connected subgraph. This generalizes the problem of finding a maximum weighted subtree of a weighted tree. \Xn{For the exact definitions of the following mentioned related polytopes, as well as of the connected blocks polytope and their connections to each other, see \Cref{sectionConsideredPolytopes}. All of theses polytopes are defined over a graph or, more general, a matroid, and most of them coincide if the underlying graph is a bridgeless \emph{cactus graph}, i.e., a graph in which two cycles have at most one vertex in common.}

In \cite{DKN15}, the \textit{connected subgraph polytope} $\CSP(G)$ of a graph $G$ is studied, where a complete inequality description for the connected subgraph polytope of a tree is provided. \Xn{As we will see in \Cref{proposition connection}, the connected subgraph polytope and the connected blocks polytope coincide if the underlying graph is a tree. In \Cref{completefacets}, we provide a complete facet description for the connected blocks polytope, generalizing the aforementioned result in \cite{DKN15}.}\Xr{Furthermore, the connected blocks polytope of a Eulerian cactus graph is unimodular equivalent to its corresponding \textit{Eulerian subgraph polytope} and to its \textit{2-edge-connected subgraph polytope}.} 

In literature, the \Xr{Eulerian subgraph polytope}\Xn{\emph{Eulerian subgraph polytope}} mostly appears as a special case of the \textit{cycle polytope} of a binary matroid, e.g., see \cite{cyclepolytope}. In this case the considered subgraphs have components that are Eulerian but these subgraphs are not necessarily connected. To the best of our knowledge, the actual Eulerian subgraph polytope -- where only Eulerian subgraphs are considered -- has not been studied so far. This can be of interest, e.g., for optimization purposes such as finding a minimum/maximum weighted Eulerian subgraph in a weighted graph $G$. Our research covers the case that the underlying graph is a Eulerian cactus graph. In this case the cycle polytope would be given by the unit cube of some dimension, but the connected blocks polytope is more complicated \Xn{and is affinely isomorphic to the actual Eulerian subgraph polytope} in the considered case\Xn{, see \Cref{proposition connection}}. In \cite{cyclepolytope}, a complete facet description of the cycle polytope of a binary matroid without some forbidden minors is given. Furthermore, the polyhedral edges are characterized and it is shown that the cycle polytope is Hirsch if the binary matroid does not have a special minor. \Xn{We obtain similar results for the connected blocks polytope which directly implies that same results for the actual Eulerian subgraph polytope hold if the underlying graph is a Eulerian cactus graph.}

The \Xr{$2$-edge-connected subgraph polytope}\Xn{\emph{$2$-edge-connected subgraph polytope}} has mostly been studied with the additional condition that the considered subgraphs are spanning subgraphs \cite{2connected1,2connected}. We want to look at the case where all $2$-edge-connected subgraphs are allowed. Therefore consider the case where the underlying graph is a $2$-edge-connected cactus graph. In this case the spanning $2$-edge-connected subgraph polytope has dimension zero since the only spanning $2$-edge-connected subgraph is the graph $G$ itself. The $2$-edge-connected subgraph polytope, as we will consider it, however, is more complex \Xn{and is affinely isomorphic to the connected blocks polytope, if the underlying graph is a $2$-edge-connected cactus graph, see \Cref{proposition connection}}. In \cite{2connected1,2connected}, several facet-defining inequalities for the spanning $2$-edge-connected subgraph polytope are studied. Furthermore, a complete inequality description is provided for some classes of graphs such as series-parallel graphs and Halin graphs. \Xn{Since we also provide a complete inequality description of the connected blocks polytope, see \Cref{completefacets}, this directly implies the same result for the $2$-edge-connected subgraph polytope, if the underlying graph is a $2$-edge-connected cactus graph.}

There exists no study of Ehrhart theoretical aspects for any of the above mentioned polytopes so far.

\subsubsection*{\bf Our Contribution and Organization}
In \Cref{sectionPreliminaries}, we give all the needed preliminaries on graph theory and introduce some notations. In \Cref{sectionConsideredPolytopes}, we formally define the connected blocks polytope and explain its relation to other polytopes of interest. We then start with some basic results on the connected blocks polytope in \Cref{sectionPolyhedralWarmUp}, where we determine the dimension of $\CBP(G)$ and characterize when connected blocks polytopes are simple and simplicial, respectively. Afterwards we study the facets in \Cref{sectionFacets}, where we give a complete facet description of $\CBP(G)$ in \Cref{completefacets}. This also yields a complete facet description for the Eulerian ($2$-edge-connected) subgraph polytope of Eulerian ($2$-edge-connected, respectively) cactus graphs and in particular generalizes the mentioned main result of \cite{DKN15}. In \Cref{sectionEdges}, motivated by \cite{cyclepolytope}, we characterize the polyhedral edges of connected blocks polytopes and show that the polytopes are Hirsch. In \Cref{sectionUnimodularTriangulations}, we initiate the study of Ehrhart theoretical aspects of connected blocks polytopes. We show that connected blocks polytopes admit a regular unimodular triangulation and are Gorenstein, which implies unimodality of its $h^\ast$-vector. Thus, these polytopes are not only interesting from an optimization point of view but also from a discrete geometry and Ehrhart theoretical point of view. We give the Ehrhart theoretical preliminaries in the beginning of that section since there is no knowledge on Ehrhart theory needed in the rest of the paper.

\section{Preliminaries}

\label{sectionPreliminaries}

By $G$ we denote an undirected graph with vertex set $V(G)$ and edge~set~$E(G)$. A connected graph $G$ is \emph{$2$-connected} (or \emph{biconnected}) if $\vert V(G) \vert \ge 3$ and $G-v$ is connected for every vertex $v \in V(G)$. Here, $G-v$ denotes the graph where $v$ and all its incident edges are deleted from~$G$. An edge $e \in E(G)$ is called a \emph{bridge} if the graph $G-e$, given by the graph with vertex set $V(G)$ and edge set $E(G) \setminus \{e\}$, has more connected components than $G$. A \emph{block} of $G$ is a maximal $2$-connected subgraph of $G$ or a bridge together with its two incident vertices, and by $\B$ we denote the set of blocks of $G$. If $H \subseteq G$ is a subgraph of $G$ whose blocks are contained in $\B$, we write $\B[H]$ for the set of blocks of $H$. For a subset of the blocks $\A \subseteq \B$ we denote by $G[\A]$ the graph that is induced by $\A$, i.e., $G[\A] \coloneqq \bigcup_{B \in \A}B$, and  with $\B \langle \A \rangle \subseteq \B$ we denote the smallest set of blocks of $\B$ that induces a connected subgraph of $G$ containing $\A$. A vertex $v \in V(G)$ is a \emph{cut vertex} if $G-v$ has more components than $G$ and by $X \subseteq V(G)$ we denote the set of cut vertices of~$G$. Note that a connected graph on at least three vertices is $2$-connected if and only if it has no cut vertices. For a connected graph $G$ its \emph{block-cut tree}, denoted by $\T\coloneqq \T(G)$, is the tree with vertex set $V(\T)=\B\cup X$ and edge set $E = \left\{ \{B,v\} \mid B \in \B, \ v \in X \cap V(B) \right\}$\Xn{, see \Cref{fig:Example} for an example}. Note that the block-cut tree is indeed a tree\footnote{\Xn{The connectivity is immediate. Suppose there is a cycle in $\T(G)$. Since all edges of $\T(G)$ are between $X$ and $\B$, there are at least two cut vertices $v_1,v_2 \in X$, that belong to this cycle. But then $v_1$ and $v_2$ also lie in a common cycle in $G$, which is a contradiction.}}. We say that $G$ is a \emph{block path of length $n$} if $G$ has $n$ blocks $\B=\{B_1,\ldots,B_n\}$ such that $B_i$ and $B_{i+1}$ have a common vertex for $i \in [n-1]$, where $[\ell]$ denotes the set $\{1,\ldots,\ell\}$, and no other pair of blocks shares a common vertex. Equivalently, $G$ is a block path if and only if its block-cut tree $\T$ is a path. A \emph{cycle} in $G$ is a subgraph $C \subseteq G$ with pairwise distinct vertices $v_1,\ldots,v_n$, $n \ge 3$, and edge set $E(C)=\{v_iv_{i+1 \bmod n} \mid i \in [n]\}$. A graph $G$ is called a \emph{cactus graph} if any two cycles have at most one vertex in common. A graph $G$ is called \emph{Eulerian} if there exists a Eulerian tour in $G$, i.e., a route starting and ending at the same vertex traversing every edge of $G$ exactly once. Equivalently, a graph is Eulerian if and only if it is connected (except for possibly some isolated vertices) and every vertex has even degree. Observe that a Eulerian cactus graph is precisely one without any bridges and only one component containing edges.

\begin{figure}
    \centering
    \begin{tikzpicture}
        \coordinate[label=center:$G$] (A) at (-1,1);
    
        \draw[red,thick] (0,0) -- (1,1);
        \draw[red,thick] (0,0) -- (2,0);
        \draw[red,thick] (2,0) -- (1,1);
        \draw[blue,thick] (2,0) -- (1,-1);
        \draw[blue,thick] (2,0) -- (3,-1);
        \draw[blue,thick] (1,-1) -- (3,-1);
        \draw[blue,thick] (2,0) -- (1.5,-1);
        \draw[blue,thick] (2,0) -- (2,-1);
        \draw[blue,thick] (2,-0.5) -- (3,-1);
        \draw[green!80!black,thick] (2,0) -- (3,1);
        \draw[green!80!black,thick] (2,0) -- (4,0);
        \draw[green!80!black,thick] (3,1) -- (4,0);
        \draw[green!80!black,thick] (3,1) -- (3,0);
        \draw[orange,thick] (4,0) -- (4,-1);
        \draw[violet,thick] (4,0) -- (5.5,0);
        \draw[yellow!80!black,thick] (5.5,0) -- (6.5,1);
        \draw[yellow!80!black,thick] (5.5,0) -- (6.5,-1);
        \draw[yellow!80!black,thick] (7.5,0) -- (6.5,1);
        \draw[yellow!80!black,thick] (7.5,0) -- (6.5,-1);
        \draw[yellow!80!black,thick] (6.5,1) -- (6.5,-1);
        \draw (6.5,-1)[brown,thick] -- (5.5,-2);
        \draw (6.5,-1)[brown,thick] -- (7.5,-2);
        \draw (7.5,-2)[brown,thick] -- (5.5,-2);
        \draw[teal,thick] (7.5,0) -- (9,0);

        \fill (0,0) circle (2pt);
        \fill (2,0) circle (2pt);
        \coordinate[label=above:$v_1$] (D) at (2,0);
        \fill (1,-1) circle (2pt);
        \fill (1,1) circle (2pt);
        \fill (4,0) circle (2pt);
        \coordinate[label=above:$v_2$] (D) at (4,0);
        \fill (3,1) circle (2pt);
        \fill (3,-1) circle (2pt);
        \fill (6.5,1) circle (2pt);
        \fill (6.5,-1) circle (2pt);
        \coordinate[label=left:$v_4$] (D) at (6.5,-1);
        \fill (7.5,0) circle (2pt);
        \coordinate[label=above:$v_5$] (D) at (7.5,0);
        \fill (9,0) circle (2pt);
        \fill (5.5,-2) circle (2pt);
        \fill (7.5,-2) circle (2pt);
        \fill (1.5,-1) circle (2pt);
        \fill (2,-1) circle (2pt);
        \fill (2,-0.5) circle (2pt);
        \fill (5.5,0) circle (2pt);
        \coordinate[label=above:$v_3$] (D) at (5.5,0);
        \fill (3,0) circle (2pt);
        \fill (4,-1) circle (2pt);

    \end{tikzpicture}
    
    \bigskip

        \begin{tikzpicture}[scale=1.2]
    \begin{scope}[yscale=1,xscale=-1,shift={(-8.5cm,0cm)}]
        \coordinate (O) at (150:0.5);
        \coordinate (P) at (150:0.5);
        \coordinate (B) at (0.25,0.433);
        \coordinate (C) at (0.25,-0.433);
        
        \fill[green!80!black!10,line width=0.5mm] (-1,0) .. controls ($(-1,0)+(45:0.95)$) and ($(0,0)+(100:0.4)$) .. (0,0);
        \fill[green!80!black!10,line width=0.5mm] (-1,0) .. controls ($(-1,0)+(-45:0.95)$) and ($(0,0)+(-100:0.4)$) .. (0,0);
        
        \draw[green!80!black,line width=0.5mm] (-1,0) .. controls ($(-1,0)+(45:0.95)$) and ($(0,0)+(100:0.4)$) .. (0,0);
        \draw[green!80!black,line width=0.5mm] (-1,0) .. controls ($(-1,0)+(-45:0.95)$) and ($(0,0)+(-100:0.4)$) .. (0,0);

        \fill[red!10,line width=0.5mm] (0,0) .. controls (105:0.95) and ($2*(B)+(150:0.4)$) .. ($2*(B)$);
        \fill[red!10,line width=0.5mm] (0,0) .. controls (15:0.95) and ($2*(B)-(150:0.4)$) .. ($2*(B)$);
        
        \draw[red,line width=0.5mm] (0,0) .. controls (105:0.95) and ($2*(B)+(150:0.4)$) .. ($2*(B)$);
        \draw[red,line width=0.5mm] (0,0) .. controls (15:0.95) and ($2*(B)-(150:0.4)$) .. ($2*(B)$);

        \fill[blue!10,line width=0.5mm] (0,0) .. controls (-15:0.95) and ($2*(C)-(210:0.4)$) .. ($2*(C)$);
        \fill[blue!10,line width=0.5mm] (0,0) .. controls (255:0.95) and ($2*(C)+(210:0.4)$) .. ($2*(C)$);
        
        \draw[blue,line width=0.5mm] (0,0) .. controls (-15:0.95) and ($2*(C)-(210:0.4)$) .. ($2*(C)$);
        \draw[blue,line width=0.5mm] (0,0) .. controls (255:0.95) and ($2*(C)+(210:0.4)$) .. ($2*(C)$);
        
        \fill[violet!10, line width=0.5mm] (-1,0) .. controls ($(-1,0)+(135:0.95)$) and ($(-2,0)+(85:0.4)$) .. (-2,0);
        \fill[violet!10, line width=0.5mm] (-1,0) .. controls ($(-1,0)+(-135:0.95)$) and ($(-2,0)-(85:0.4)$) .. (-2,0);
        
        \draw[violet, line width=0.5mm] (-1,0) .. controls ($(-1,0)+(135:0.95)$) and ($(-2,0)+(85:0.4)$) .. (-2,0);
        \draw[violet, line width=0.5mm] (-1,0) .. controls ($(-1,0)+(-135:0.95)$) and ($(-2,0)-(85:0.4)$) .. (-2,0);

        \fill[orange!10, line width=0.5mm] (-1,0) .. controls ($(-1,0)+(230:0.95)$) and ($(-1,-1)+(180:0.4)$) .. (-1,-1);
        \fill[orange!10, line width=0.5mm] (-1,0) .. controls ($(-1,0)+(-50:0.95)$) and ($(-1,-1)+(0:0.4)$) .. (-1,-1);
        
        \draw[orange, line width=0.5mm] (-1,0) .. controls ($(-1,0)+(230:0.95)$) and ($(-1,-1)+(180:0.4)$) .. (-1,-1);
        \draw[orange, line width=0.5mm] (-1,0) .. controls ($(-1,0)+(-50:0.95)$) and ($(-1,-1)+(0:0.4)$) .. (-1,-1);

        \fill[yellow!80!black!10,line width=0.5mm] (-2.5,0) circle (0.5);
        \fill[teal!10,line width=0.5mm] (-3.5,0) circle (0.5);
        \fill[brown!10,line width=0.5mm] (-2.5,-1) circle (0.5);
        
        \draw[yellow!80!black,line width=0.5mm] (-2.5,0) circle (0.5);
        \draw[teal,line width=0.5mm] (-3.5,0) circle (0.5);
        \draw[brown,line width=0.5mm] (-2.5,-1) circle (0.5);
        \fill (0,0) circle (2pt);
        \fill (-1,0) circle (2pt);
        \fill (-2,0) circle (2pt);
        \fill (-3,0) circle (2pt);
        \fill (-2.5,-0.5) circle (2pt);
        \coordinate[label=center:\small{$B_1$}] (A1) at (B);
        \coordinate[label=center:\small{$B_6$}] (A1) at (C);
        \coordinate[label=center:\small{$B_2$}] (A1) at (-0.5,0);
        \coordinate[label=center:\small{$B_3$}] (A1) at (-1.5,0);
        \coordinate[label=center:\small{$B_4$}] (A1) at (-2.5,0);
        \coordinate[label=center:\small{$B_5$}] (A1) at (-3.5,0);
        \coordinate[label=center:\small{$B_8$}] (A1) at (-2.5,-1);
        \coordinate[label=center:\small{$B_7$}] (A1) at (-1,-0.5);
        \end{scope}
    \end{tikzpicture}
    
    \bigskip
    
    \begin{tikzpicture}
        \coordinate[label=center:$\T(G)$] (A) at (-2,1);

        \coordinate (A) at (45:0.13);
    
        \draw (-0.7,0.7) -- (0,0);
        \draw (-0.7,-0.7) -- (0,0);
        \draw (0,0) -- (1,0);
        \draw (1,0) -- (2,0);
        \draw (2,0) -- (2,-0.7);
        \draw (2,0) -- (3,0);
        \draw (3,0) -- (4,0);
        \draw (4,0) -- (5,0);
        \draw (5,0) -- (5.7,0.7);
        \draw (5,0) -- (5.7,-0.7);
        \draw (5.7,0.7) -- (6.7,0.7);
        \draw (5.7,-0.7) -- (6.7,-0.7);
        
        \fill[red] ($(-0.7,0.7)+(A)$) rectangle ($(-0.7,0.7)-(A)$);
        \coordinate[label=above:$B_1$] (D) at (-0.7,0.7);
        \fill[blue] ($(-0.7,-0.7)+(A)$) rectangle ($(-0.7,-0.7)-(A)$);
        \coordinate[label=below:$B_6$] (D) at (-0.7,-0.7);
        \fill (0,0) circle (2pt);
        \coordinate[label=above:$v_1$] (D) at (0,0);
        \fill[green!80!black] ($(1,0)+(A)$) rectangle ($(1,0)-(A)$);
        \coordinate[label=below:$B_2$] (D) at (1,0);
        \fill (2,0) circle (2pt);
        \coordinate[label=above:$v_2$] (D) at (2,0);
        \fill[orange] ($(2,-0.7)+(A)$) rectangle ($(2,-0.7)-(A)$);
        \coordinate[label=below:$B_7$] (D) at (2,-0.7);
        \fill[violet] ($(3,0)+(A)$) rectangle ($(3,0)-(A)$);
        \coordinate[label=below:$B_3$] (D) at (3,0);
        \fill (4,0) circle (2pt);
        \coordinate[label=above:$v_3$] (D) at (4,0);
        \fill[yellow!80!black] ($(5,0)+(A)$) rectangle ($(5,0)-(A)$);
        \coordinate[label=below:$B_4$] (D) at (5,0);
        \fill (5.7,0.7) circle (2pt);
        \coordinate[label=above:$v_5$] (D) at (5.7,0.7);
        \fill (5.7,-0.7) circle (2pt);
        \coordinate[label=below:$v_4$] (D) at (5.7,-0.7);
        \fill[teal] ($(6.7,0.7)+(A)$) rectangle ($(6.7,0.7)-(A)$);
        \coordinate[label=above:$B_5$] (D) at (6.7,0.7);
        \fill[brown] ($(6.7,-0.7)+(A)$) rectangle ($(6.7,-0.7)-(A)$);
        \coordinate[label=below:$B_8$] (D) at (6.7,-0.7);
    \end{tikzpicture}
    \caption{Example for a graph $G$, its block structure, and its block-cut tree $\T(G)$.}
    \label{fig:Example}
\end{figure}
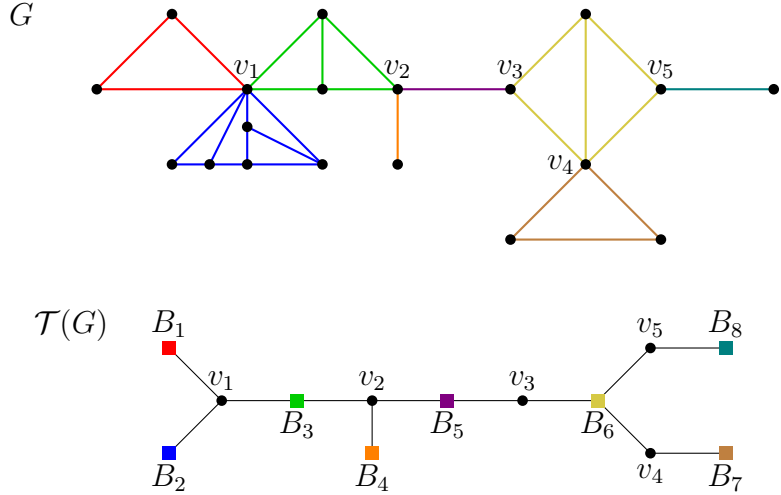

\section{Considered polytopes}

\label{sectionConsideredPolytopes}

\subsection{Connected Blocks Polytope}

Let $G$ be a connected graph and $\B$ its set of blocks. For a subset of the blocks $\A \subseteq \B$ we define the incidence vector $\chi^\A \in \{0,1\}^{\B}$ by
\begin{align*}
    \chi_B^\A\coloneqq \begin{cases} 1, \text{ if } B \in \A, \\
                            0, \text{ otherwise}.\end{cases}
\end{align*}
We are interested in finding subsets of blocks that induce a connected graph. We define the \textit{connected blocks polytope} of $G$ by
\begin{align*}
    \CBP(G) \coloneqq \conv\{\chi^\A \mid \A \subseteq \B, \ G[\A] \text{ is connected} \}.
\end{align*}
We also consider the empty graph to be connected and hence, $\chi^\emptyset=\mathbf{0} \in \CBP(G)$. Given the block-cut tree $\T$ of $G$ and $\A \subseteq \B \subseteq V(\T)$, we let $\T\langle\A\rangle \subseteq \T$ denote the smallest subtree containing $\A$. Then $G[\A]$ is connected if and only if $V(\T\langle\A\rangle) \cap \B=\A$. Thus, we get
\begin{align*}
    \CBP(G) = \conv\{\chi^\A \mid \A \subseteq \B, \ V(\T\langle\A\rangle) \cap \B=\A \}.
\end{align*}
In the following, we will show that $\CBP(G)$ is closely related to several polytopes from literature.

\subsection{Connected Subgraph Polytope}

The \textit{connected subgraph polytope} $\CSP(G)$ of a graph~$G$ is defined as
\begin{align*}
    \CSP(G) \coloneqq \conv\{\chi^{E(H)} \mid H \subseteq G \text{ is connected} \},
\end{align*}
where $\chi^{E(H)} \in \{0,1\}^{E(G)}$ for a subgraph $H \subseteq G$ is defined as
\begin{align*}
         \chi^{E(H)}_e\coloneqq \begin{cases} 1, \text{ if } e \in E(H), \\
                               0, \text{ otherwise}.\end{cases}
\end{align*}
If $T$ is a tree, every block of $T$ consists of a single edge $e \in E(T)$ together with its incident vertices. We thus have a one-to-one correspondence between the blocks of $T$ and the edges of $T$ and it directly follows that
\begin{align*}
    \CBP(T)=\CSP(T)
\end{align*}
for every tree $T$. Hence, every connected subgraph polytope of a tree is a connected blocks polytope. The converse is not true since in a tree every block contains at most two cut vertices, while this does not hold for arbitrary graphs.

In \cite{DKN15}, a complete inequality description is provided for $\CSP(G)$ in the case that $G$ is a tree, a cycle, or has no matching of cardinality three. In \Cref{completefacets}, we generalize the first case (\cite[Theorem 6]{DKN15}) by giving a complete facet description for $\CBP(G)$ for general graphs, i.e., the connected blocks polytopes that are not necessarily given by the connected subgraph polytope of a tree.

\subsection{Eulerian Subgraph Polytope}

For a graph $G$ the \textit{Eulerian subgraph polytope} is defined as
\begin{align}
    \label{defESP}
    \mathrm{ESP}(G) \coloneqq \conv\{ \chi^{E(H)} \mid H \subseteq G \text{ is Eulerian} \}.
\end{align}
In the literature, the term Eulerian subgraph polytope is also used for the cycle polytope of a graphic matroid; see for example \cite{cyclepolytope} for a definition and more details. However, that definition also allows subgraphs that have multiple non-trivial components (each of which is Eulerian). Thus those graphs are not really Eulerian and the naming is misleading. Our definition in \eqref{defESP} is thus different but arguably more sound in terms of graph theory.

Suppose now that $G$ is a connected Eulerian cactus graph. Then the blocks of $G$ are exactly the cycles of $G$. Given a cycle $C$ of $G$, every Eulerian subgraph $H$ of $G$ either contains all edges of $C$ or none, since otherwise there would exist a bridge in $H$. Hence, $x_e=x_f$ for all $x \in \mathrm{ESP}(G)$ and all $e,f \in E(C)$. Thus it is not necessary to have one variable per edge in $G$, but it suffices to have one variable for each cycle in $G$. For a Eulerian cactus graph $G$, let $\C \coloneqq \C(G)$ denote the set of cycles of $G$. The previous discussion gives rise to what we call the \textit{projected Eulerian subgraph polytope} of a Eulerian cactus graph $G$. Namely,
\begin{align*}
    \mathrm{ESP}_p(G)\coloneqq \conv\{\chi^{\C(H)} \mid H \subseteq G \text{ Eulerian} \},
\end{align*}
where $\chi^{\C(H)} \in \{0,1\}^\C$ is defined by
\begin{align*}
    \chi^{\C(H)}_C \coloneqq \begin{cases} 1, \text{ if } C \in \C(H), \\
                                   0, \text{ otherwise.}\end{cases}
\end{align*}
By the above discussion, the two polytopes $\ESP_p(G)$ and $\ESP(G)$ are unimodular equivalent. Since every block of a Eulerian cactus graph is a cycle and every Eulerian subgraph of a Eulerian cactus graph is precisely a subset of cycles whose union is connected, it directly follows that
\begin{align*}
    \CBP(G)=\ESP_p(G)
\end{align*}
for every Eulerian cactus graph $G$.

As already mentioned, the cycle polytope of a graphic matroid has been studied in literature. \cite{cyclepolytope} provides a nonredundant linear system of inequalities describing the cycle polytope of a graphic matroid. If $G$ is a graph such that every subgraph whose components are Eulerian already is connected except possibly some isolated vertices, the cycle polytope and the Eulerian subgraph polytope defined in \eqref{defESP} coincide and Theorem~4.23 of \cite{cyclepolytope} is true in this case. Note that this is only the case if and only if all two edge-disjoint cycles of $G$ share a common vertex. To the best of our knowledge, there is no work on the Eulerian subgraph polytope as defined in \eqref{defESP}. We start to close this gap by considering the case that the underlying graph is a Eulerian cactus graph. Note that the cycle polytope of a Eulerian cactus graph would be unimodular equivalent to the trivial unit cube $Q_n\coloneqq [0,1]^n$ of dimension~$n$, where $n$ is given by the number of cycles of $G$. In our case, however, we typically deal with strict subpolytopes of $Q_n$.

\subsection{2-edge-connected Subgraph Polytope}

We say that $G$ is \textit{$2$-edge-connected} if $G$ is connected and contains no bridges. We define the \textit{$2$-edge-connected subgraph polytope} by
\begin{align}
    \label{defTECSP}
    \mathrm{TECSP}(G) \coloneqq \conv\{ \chi^{E(H)} \mid H \subseteq G \text{ is } 2\text{-edge-connected} \}.
\end{align}
In the literature, the $2$-edge-connected subgraph polytope is typically only considered with the additional restriction that the subgraphs need to be spanning, see \cite{2connected}, i.e., the subgraphs must contain all vertices of the underlying graph. As there is no work on the $2$-edge-connected subgraph polytope as defined in \eqref{defTECSP}, we start to close this gap by considering the case that the underlying graph is a $2$-edge-connected cactus graph. Note that the spanning $2$-edge-connected subgraph polytope of a $2$-edge-connected cactus graph $G$ would be $\{\chi^{E(G)}\}=\{\mathbf{1}\}$, since $G$ is the only spanning $2$-edge-connected subgraph of $G$. In our case, however, the polytope is much more interesting.

Suppose that $G$ is a $2$-edge-connected cactus graph. By definition a connected cactus graph is $2$-edge-connected if it contains no bridges, i.e., if and only if it is Eulerian. Analogously to the Eulerian subgraph polytope of a Eulerian cactus graph we observe that for $2$-edge-connected cactus graphs the $2$-edge-connected subgraph polytope is unimodular equivalent to the connected blocks polytope.

\bigskip

Summing up, we obtain the following relations, where $\P_1 \cong \P_2$ denotes that the two polytopes $\P_1$ and $\P_2$ are unimodular equivalent:
\begin{proposition}
    \label{proposition connection}
    \begin{enumerate}[(i)]
        \item Let $T$ be a tree. Then we have
            \begin{align*}
                \CBP(T)=\CSP(T).
            \end{align*}
        \item Let $G$ be a connected Eulerian cactus graph. Then we have
            \begin{align*}
                \ESP_p(G)=\CBP(G) \cong \TECSP(G)=\ESP(G).
            \end{align*}
    \end{enumerate}
\end{proposition}

Thus, in the following, it will suffice to consider the general connected blocks polytope, to directly obtain results on the other above mentioned polytopes.

\section{Polyhedral Warm-up}

\label{sectionPolyhedralWarmUp}

In this section we determine the dimension of the connected blocks polytope and characterize when those polytopes are simple and simplicial, respectively.

\begin{proposition}
    \label{dimension}
    Let $G$ be a connected graph and $\B$ its set of blocks. Then we have
    \begin{align*}
        \dim(\CBP(G))=\vert \B \vert.
    \end{align*}
    In particular, $\CBP(G)$ is full-dimensional.
\end{proposition}

\begin{proof}
    The $\vert \B \vert +1$ vectors $\{\chi^{\{B\}}\}_{B \in \B} \cup \{\mathbf{0}\}$ are affinely independent vertices of $\CBP(G)$.
\end{proof}

\begin{proposition}
    \label{propsimplesimplicial}
    Let $G$ be a connected graph and $\B$ its set of blocks.
    \begin{enumerate}[(i)]
        \item $\CBP(G)$ is simple if and only if $G$ has at most one cut vertex.
        \item $\CBP(G)$ is simplicial if and only if $\vert \B \vert \le 2$.
    \end{enumerate}
    
\end{proposition}

\begin{proof}
    \begin{enumerate}[$(i)$]
        \item The polytope is trivial if $G$ has no cut vertex. Suppose $G$ has precisely one cut vertex $v \in V(G)$. Then $\CBP(G)$ is given by the $\vert \B \vert$-dimensional unit cube since every subset $\A \subseteq \B$ induces a connected subgraph of $G$. In particular, $\CBP(G)$ is simple.
    
        Suppose now that there are at least two distinct cut vertices $v,w \in V(G)$. Since these two vertices are connected in $G$ by a block path of length $\ge 1$, there must exist a block path $P_3=B_1 \cup B_2 \cup B_3 \subseteq G$ of length $3$. Consider the face $\F$ of $\CBP(G)$ given by 
        \begin{align*}
            \F \coloneqq \{ x\in \CBP(G) \mid x_B=0 \text{ for all } B \in \B \backslash \{B_1,B_2,B_3\}\}.
        \end{align*}
        Then $\F$ is unimodular equivalent to the $3$-dimensional connected blocks polytope of~$P_3$. It is easy to verify that the vertex $\chi^{\{B_1\}}$ of $\CBP(P_3)$ lies in an edge with each of the vertices $\chi^\emptyset, \chi^{\{B_3\}}, \chi^{\{B_1,B_2\}}$ and $\chi^{\{B_1,B_2,B_3\}}$, see also \Cref{sectionEdges}. Thus, by Proposition 2.4 of \cite{Ziegler} the vertex figure of $\chi^{\{B_1\}}$ is not a simplex, $\F$ is not simple, and thus $\CBP(G)$ is not simple.
        
        \item For $\vert \B \vert \in \{1,2\}$, we have $\CBP(G)=Q_1$ or $\CBP(G)=Q_2$, respectively, which are simplicial.
        
        Suppose $\vert \B \vert \ge 3$. Let $B_1$ and $B_2$ be blocks sharing a common vertex. Since $\vert \B \vert \ge 3$,
            \begin{align*}
                \F \coloneqq \{x \in \CBP(G) \mid x_B=0 \text{ for all } B \in \B \backslash \{B_1,B_2\} \}
          \end{align*}
        defines a proper face of $\CBP(G)$. Since $\F$ is affinely isomorphic to $Q_2$, it follows that $\CBP(G)$ is not simplicial. \qedhere
    \end{enumerate}
    \end{proof}

\section{Facets}

\label{sectionFacets}

In this section we will provide a complete and irredundant facet description for the connected blocks polytope.

\subsection*{Box inequalities}

We first show that the box inequalities $0 \le x_B \le 1$ for $B \in \B$ define facets of $\CBP(G)$.

\begin{theorem}
    \label{theoremBox}
    Let $G$ be a connected graph, $\B$ its set of blocks and $B \in \B$. Then the inequalities
    \begin{align*}
        x_B \ge 0 \text{ and } x_B \le 1
    \end{align*}
    define facets of $\CBP(G)$.
\end{theorem}

\begin{proof}
    Since $\CBP(G)$ is a $0$-$1$-polytope, it is clear that the inequalities define supporting hyperplanes of $\CBP(G)$. Let $\mathcal{F}$ be the face of $\CBP(G)$ defined by $x_B=0$, i.e.,
    \begin{align*}
        \F \coloneqq \{x \in \CBP(G) \mid x_B=0 \}.
    \end{align*}
    Then the $\vert \B \vert$ many vertices $\{\chi^{\{B'\}}\}_{B' \in \B \setminus \{B\}} \cup \{\textbf{0}\}$ lie in $\F$ and are affinely independent. It follows that $\F$ is a facet, since $\dim(\CBP(G))=\vert \B \vert$.

    Now consider $x_B \le 1$ and let $\B=\{B_1,\ldots,B_n\}$. After possibly reordering the blocks we can assume that $B=B_1$ and $G[\{B_1,\ldots,B_i\}]$ is connected for all $i \in [n]$. Then the incidence vectors $\chi^{\{B_1,\ldots,B_i\}}$ are vertices of the face defined by $x_B=1$ for all $i \in [n]$. Since these $\vert \B \vert$ many vectors are affinely independent, the inequality $x_B \le 1$ defines a facet of $\CBP(G)$.
\end{proof}

\subsection*{Independent blocks inequalities}

We will motivate the next type of inequalities by a few examples. Consider a block path $P_3=B_1 \cup B_2 \cup B_3$ of length $3$\Xn{, which can also be seen as a subgraph of the graph in \Cref{fig:Example}}. \Xr{We have to ensure that it is not possible to choose both $B_1$ and $B_3$ without choosing $B_2$.}\Xn{Since the blocks $B_1$ and $B_3$ do not share a common vertex, we have to find an inequality that is feasible for $\CBP(P_3)$, but not for $\chi^{\{B_1,B_3\}}$.} This can be achieved by imposing the inequality
\begin{align*}
    x_{B_1}-x_{B_2}+x_{B_3} \le 1.
\end{align*}
\Xr{since this forbids the incidence vector $\chi^{\{B_1,B_3\}}$.} For an arbitrary connected graph $G$ with set of blocks $\B$ this can be generalized to
\begin{align}
    \label{ILP}
    x_{B_1}-x_{B_2}+x_{B_3} \le 1,
\end{align}
for any $B_1,B_2,B_3 \in \B$, with $V(B_1) \cap V(B_3)=\emptyset$ and  $B_2 \in \B \langle\{B_1,B_3\}\rangle\setminus \{B_1,B_3\}$. It is easy to check that together with the box inequalities this already describes all feasible lattice points
\begin{align*}
    \CBP(G) \cap \Z^\B.
\end{align*}
However, these inequalities do not suffice to describe $\CBP(G)$, i.e., the convex hull of these points. For example consider a block path $P_5=B_1\cup B_2\cup B_3\cup B_4\cup B_5$ of length $5$\Xn{, which again can be seen as a subgraph of the graph in \Cref{fig:Example}}. Then the inequality
\begin{align*}
    x_{B_1}-x_{B_2}+x_{B_3}-x_{B_4}+x_{B_5} \le 1
\end{align*}
is valid for $\CBP(P_5)$ and forbids the point $(\frac{1}{2},0,\frac{1}{2},0,\frac{1}{2})^T$, which would remain feasible if only the inequalities in \eqref{ILP} and the box inequalities were imposed. Similarly, if $G$ is a graph with blocks $\B=\{B_1,B_2,B_3,B\}$ such that $V(B) \cap V(B_i) \neq \emptyset$ for all $i \in [3]$ and $V(B_i) \cap V(B_j)=\emptyset$ for $i,j \in [3]$, $i\neq j$, then also the inequality
\begin{align*}
    x_{B_1}+x_{B_2}+x_{B_3}-2x_{B} \le 1
\end{align*}
defines a supporting hyperplane of $\CBP(G)$ and cuts off the otherwise feasible point $(\frac{1}{2},\frac{1}{2},\frac{1}{2},0)^T$. \Xn{(For the graph in \Cref{fig:Example} a subgraph of this type would be given by $B_3 \cup B_4 \cup B_5 \cup B_8$, where $B=B_4$.)}

We now generalize these examples as follows. For a connected graph $G$ and its set of blocks~$\B$ we call $\I \subseteq \B$ an \textit{independent set of blocks} if there are no two blocks in $\I$ that share a common vertex. \Xn{Equivalently, all blocks of $\I$ have distance at least four to each other in the block-cut tree $\T(G)$.} If $\A \subseteq \B$ is a set of blocks such that $G[\A]$ is connected and $\A$ contains an independent set of blocks $\I \subseteq \A$, then $\A$ also contains all blocks of $\B\langle \I \rangle$. On the other hand, if $\A$ does not induce a connected subgraph, then there always exists an independent set of blocks $\I \subseteq \B$, such that $\I \subseteq \A$ but $\B\langle\I\rangle \nsubseteq \A$. This motivates the \textit{independent blocks inequalities}, defined as follows.

\begin{definition}(Independent Blocks Inequality)
    Let $G$ be a connected graph, $\B$ its set of blocks, and $\I \subseteq \B$ an independent set of blocks. We call
    \begin{align}
        \label{ineqindcyc}
        \sum_{B \in \B}\alpha_Bx_B \le 1
    \end{align}
    an \emph{independent blocks inequality} if
    \begin{itemize}
        \item for $B \notin \B\langle \I \rangle$ we have $\alpha_B=0$,
        \item for $B \in \I$ we have $\alpha_B=1$, and
        \item for $B \in \B\langle \I \rangle \setminus \I$ we have $\alpha_B \in \Z_{\le 0}$,
    \end{itemize}
    such that
    \begin{itemize}
        \item $\sum_{B \in \B\langle\I\rangle \backslash \I}\alpha_B=-(\vert \I \vert-1)$, and
        \item for all $I \subseteq \I$ we have
        \begin{align}
            \sum_{B \in \B\langle I\rangle\backslash \I}\alpha_B \le -(\vert I \vert-1).
            \label{condIBI}
        \end{align}
    \end{itemize}
We will also denote an independent blocks inequality by the pair $(\I,\alpha)$.
\end{definition}

Observe that for $\I=\{B\} \subseteq \B$ we get the box inequality $x_B \le 1$.

\Xn{\begin{example}
    Consider again the graph $G$ of \Cref{fig:Example}. We will give some examples of independent blocks inequalities of $G$ here, by writing the coefficient $\alpha_B$ in the corresponding block $B$.

\begin{center}
    \begin{tikzpicture}[scale=1.2]
    \begin{scope}[yscale=1,xscale=-1,shift={(-8.5cm,0cm)}]
        \coordinate (O) at (150:0.5);
        \coordinate (P) at (150:0.5);
        \coordinate (B) at (0.25,0.433);
        \coordinate (C) at (0.25,-0.433);
        
        \fill[green!80!black!10,line width=0.5mm] (-1,0) .. controls ($(-1,0)+(45:0.95)$) and ($(0,0)+(100:0.4)$) .. (0,0);
        \fill[green!80!black!10,line width=0.5mm] (-1,0) .. controls ($(-1,0)+(-45:0.95)$) and ($(0,0)+(-100:0.4)$) .. (0,0);
        
        \draw[green!80!black,line width=0.5mm] (-1,0) .. controls ($(-1,0)+(45:0.95)$) and ($(0,0)+(100:0.4)$) .. (0,0);
        \draw[green!80!black,line width=0.5mm] (-1,0) .. controls ($(-1,0)+(-45:0.95)$) and ($(0,0)+(-100:0.4)$) .. (0,0);

        \fill[red!10,line width=0.5mm] (0,0) .. controls (105:0.95) and ($2*(B)+(150:0.4)$) .. ($2*(B)$);
        \fill[red!10,line width=0.5mm] (0,0) .. controls (15:0.95) and ($2*(B)-(150:0.4)$) .. ($2*(B)$);
        
        \draw[red,line width=0.5mm] (0,0) .. controls (105:0.95) and ($2*(B)+(150:0.4)$) .. ($2*(B)$);
        \draw[red,line width=0.5mm] (0,0) .. controls (15:0.95) and ($2*(B)-(150:0.4)$) .. ($2*(B)$);

        \fill[blue!10,line width=0.5mm] (0,0) .. controls (-15:0.95) and ($2*(C)-(210:0.4)$) .. ($2*(C)$);
        \fill[blue!10,line width=0.5mm] (0,0) .. controls (255:0.95) and ($2*(C)+(210:0.4)$) .. ($2*(C)$);
        
        \draw[blue,line width=0.5mm] (0,0) .. controls (-15:0.95) and ($2*(C)-(210:0.4)$) .. ($2*(C)$);
        \draw[blue,line width=0.5mm] (0,0) .. controls (255:0.95) and ($2*(C)+(210:0.4)$) .. ($2*(C)$);
        
        \fill[violet!10, line width=0.5mm] (-1,0) .. controls ($(-1,0)+(135:0.95)$) and ($(-2,0)+(85:0.4)$) .. (-2,0);
        \fill[violet!10, line width=0.5mm] (-1,0) .. controls ($(-1,0)+(-135:0.95)$) and ($(-2,0)-(85:0.4)$) .. (-2,0);
        
        \draw[violet, line width=0.5mm] (-1,0) .. controls ($(-1,0)+(135:0.95)$) and ($(-2,0)+(85:0.4)$) .. (-2,0);
        \draw[violet, line width=0.5mm] (-1,0) .. controls ($(-1,0)+(-135:0.95)$) and ($(-2,0)-(85:0.4)$) .. (-2,0);

        \fill[orange!10, line width=0.5mm] (-1,0) .. controls ($(-1,0)+(230:0.95)$) and ($(-1,-1)+(180:0.4)$) .. (-1,-1);
        \fill[orange!10, line width=0.5mm] (-1,0) .. controls ($(-1,0)+(-50:0.95)$) and ($(-1,-1)+(0:0.4)$) .. (-1,-1);
        
        \draw[orange, line width=0.5mm] (-1,0) .. controls ($(-1,0)+(230:0.95)$) and ($(-1,-1)+(180:0.4)$) .. (-1,-1);
        \draw[orange, line width=0.5mm] (-1,0) .. controls ($(-1,0)+(-50:0.95)$) and ($(-1,-1)+(0:0.4)$) .. (-1,-1);

        \fill[yellow!80!black!10,line width=0.5mm] (-2.5,0) circle (0.5);
        \fill[teal!10,line width=0.5mm] (-3.5,0) circle (0.5);
        \fill[brown!10,line width=0.5mm] (-2.5,-1) circle (0.5);
        
        \draw[yellow!80!black,line width=0.5mm] (-2.5,0) circle (0.5);
        \draw[teal,line width=0.5mm] (-3.5,0) circle (0.5);
        \draw[brown,line width=0.5mm] (-2.5,-1) circle (0.5);
        \fill (0,0) circle (2pt);
        \fill (-1,0) circle (2pt);
        \fill (-2,0) circle (2pt);
        \fill (-3,0) circle (2pt);
        \fill (-2.5,-0.5) circle (2pt);
        \coordinate[label=center:\small{$1$}] (A1) at (B);
        \coordinate[label=center:\small{$0$}] (A1) at (C);
        \coordinate[label=center:\small{$-1$}] (A1) at (-0.5,0);
        \coordinate[label=center:\small{$0$}] (A1) at (-1.5,0);
        \coordinate[label=center:\small{$0$}] (A1) at (-2.5,0);
        \coordinate[label=center:\small{$0$}] (A1) at (-3.5,0);
        \coordinate[label=center:\small{$1$}] (A1) at (-2.5,-1);
        \coordinate[label=center:\small{$0$}] (A1) at (-1,-0.5);
        \coordinate[label=center:\small{$\le 1$,}] (A1) at (-4.5,0);
        \end{scope}
    \end{tikzpicture} \begin{tikzpicture}[scale=1.2]
    \begin{scope}[yscale=1,xscale=-1,shift={(-8.5cm,0cm)}]
        \coordinate (O) at (150:0.5);
        \coordinate (P) at (150:0.5);
        \coordinate (B) at (0.25,0.433);
        \coordinate (C) at (0.25,-0.433);
        
        \fill[green!80!black!10,line width=0.5mm] (-1,0) .. controls ($(-1,0)+(45:0.95)$) and ($(0,0)+(100:0.4)$) .. (0,0);
        \fill[green!80!black!10,line width=0.5mm] (-1,0) .. controls ($(-1,0)+(-45:0.95)$) and ($(0,0)+(-100:0.4)$) .. (0,0);
        
        \draw[green!80!black,line width=0.5mm] (-1,0) .. controls ($(-1,0)+(45:0.95)$) and ($(0,0)+(100:0.4)$) .. (0,0);
        \draw[green!80!black,line width=0.5mm] (-1,0) .. controls ($(-1,0)+(-45:0.95)$) and ($(0,0)+(-100:0.4)$) .. (0,0);

        \fill[red!10,line width=0.5mm] (0,0) .. controls (105:0.95) and ($2*(B)+(150:0.4)$) .. ($2*(B)$);
        \fill[red!10,line width=0.5mm] (0,0) .. controls (15:0.95) and ($2*(B)-(150:0.4)$) .. ($2*(B)$);
        
        \draw[red,line width=0.5mm] (0,0) .. controls (105:0.95) and ($2*(B)+(150:0.4)$) .. ($2*(B)$);
        \draw[red,line width=0.5mm] (0,0) .. controls (15:0.95) and ($2*(B)-(150:0.4)$) .. ($2*(B)$);

        \fill[blue!10,line width=0.5mm] (0,0) .. controls (-15:0.95) and ($2*(C)-(210:0.4)$) .. ($2*(C)$);
        \fill[blue!10,line width=0.5mm] (0,0) .. controls (255:0.95) and ($2*(C)+(210:0.4)$) .. ($2*(C)$);
        
        \draw[blue,line width=0.5mm] (0,0) .. controls (-15:0.95) and ($2*(C)-(210:0.4)$) .. ($2*(C)$);
        \draw[blue,line width=0.5mm] (0,0) .. controls (255:0.95) and ($2*(C)+(210:0.4)$) .. ($2*(C)$);
        
        \fill[violet!10, line width=0.5mm] (-1,0) .. controls ($(-1,0)+(135:0.95)$) and ($(-2,0)+(85:0.4)$) .. (-2,0);
        \fill[violet!10, line width=0.5mm] (-1,0) .. controls ($(-1,0)+(-135:0.95)$) and ($(-2,0)-(85:0.4)$) .. (-2,0);
        
        \draw[violet, line width=0.5mm] (-1,0) .. controls ($(-1,0)+(135:0.95)$) and ($(-2,0)+(85:0.4)$) .. (-2,0);
        \draw[violet, line width=0.5mm] (-1,0) .. controls ($(-1,0)+(-135:0.95)$) and ($(-2,0)-(85:0.4)$) .. (-2,0);

        \fill[orange!10, line width=0.5mm] (-1,0) .. controls ($(-1,0)+(230:0.95)$) and ($(-1,-1)+(180:0.4)$) .. (-1,-1);
        \fill[orange!10, line width=0.5mm] (-1,0) .. controls ($(-1,0)+(-50:0.95)$) and ($(-1,-1)+(0:0.4)$) .. (-1,-1);
        
        \draw[orange, line width=0.5mm] (-1,0) .. controls ($(-1,0)+(230:0.95)$) and ($(-1,-1)+(180:0.4)$) .. (-1,-1);
        \draw[orange, line width=0.5mm] (-1,0) .. controls ($(-1,0)+(-50:0.95)$) and ($(-1,-1)+(0:0.4)$) .. (-1,-1);

        \fill[yellow!80!black!10,line width=0.5mm] (-2.5,0) circle (0.5);
        \fill[teal!10,line width=0.5mm] (-3.5,0) circle (0.5);
        \fill[brown!10,line width=0.5mm] (-2.5,-1) circle (0.5);
        
        \draw[yellow!80!black,line width=0.5mm] (-2.5,0) circle (0.5);
        \draw[teal,line width=0.5mm] (-3.5,0) circle (0.5);
        \draw[brown,line width=0.5mm] (-2.5,-1) circle (0.5);
        \fill (0,0) circle (2pt);
        \fill (-1,0) circle (2pt);
        \fill (-2,0) circle (2pt);
        \fill (-3,0) circle (2pt);
        \fill (-2.5,-0.5) circle (2pt);
        \coordinate[label=center:\small{$0$}] (A1) at (B);
        \coordinate[label=center:\small{$0$}] (A1) at (C);
        \coordinate[label=center:\small{$0$}] (A1) at (-0.5,0);
        \coordinate[label=center:\small{$1$}] (A1) at (-1.5,0);
        \coordinate[label=center:\small{$-2$}] (A1) at (-2.5,0);
        \coordinate[label=center:\small{$1$}] (A1) at (-3.5,0);
        \coordinate[label=center:\small{$1$}] (A1) at (-2.5,-1);
        \coordinate[label=center:\small{$0$}] (A1) at (-1,-0.5);
        \coordinate[label=center:\small{$\le 1$,}] (A1) at (-4.5,0);
        \end{scope}
    \end{tikzpicture}
    \end{center}

    \begin{center}
    \begin{tikzpicture}[scale=1.2]
    \begin{scope}[yscale=1,xscale=-1,shift={(-8.5cm,0cm)}]
        \coordinate (O) at (150:0.5);
        \coordinate (P) at (150:0.5);
        \coordinate (B) at (0.25,0.433);
        \coordinate (C) at (0.25,-0.433);
        
        \fill[green!80!black!10,line width=0.5mm] (-1,0) .. controls ($(-1,0)+(45:0.95)$) and ($(0,0)+(100:0.4)$) .. (0,0);
        \fill[green!80!black!10,line width=0.5mm] (-1,0) .. controls ($(-1,0)+(-45:0.95)$) and ($(0,0)+(-100:0.4)$) .. (0,0);
        
        \draw[green!80!black,line width=0.5mm] (-1,0) .. controls ($(-1,0)+(45:0.95)$) and ($(0,0)+(100:0.4)$) .. (0,0);
        \draw[green!80!black,line width=0.5mm] (-1,0) .. controls ($(-1,0)+(-45:0.95)$) and ($(0,0)+(-100:0.4)$) .. (0,0);

        \fill[red!10,line width=0.5mm] (0,0) .. controls (105:0.95) and ($2*(B)+(150:0.4)$) .. ($2*(B)$);
        \fill[red!10,line width=0.5mm] (0,0) .. controls (15:0.95) and ($2*(B)-(150:0.4)$) .. ($2*(B)$);
        
        \draw[red,line width=0.5mm] (0,0) .. controls (105:0.95) and ($2*(B)+(150:0.4)$) .. ($2*(B)$);
        \draw[red,line width=0.5mm] (0,0) .. controls (15:0.95) and ($2*(B)-(150:0.4)$) .. ($2*(B)$);

        \fill[blue!10,line width=0.5mm] (0,0) .. controls (-15:0.95) and ($2*(C)-(210:0.4)$) .. ($2*(C)$);
        \fill[blue!10,line width=0.5mm] (0,0) .. controls (255:0.95) and ($2*(C)+(210:0.4)$) .. ($2*(C)$);
        
        \draw[blue,line width=0.5mm] (0,0) .. controls (-15:0.95) and ($2*(C)-(210:0.4)$) .. ($2*(C)$);
        \draw[blue,line width=0.5mm] (0,0) .. controls (255:0.95) and ($2*(C)+(210:0.4)$) .. ($2*(C)$);
        
        \fill[violet!10, line width=0.5mm] (-1,0) .. controls ($(-1,0)+(135:0.95)$) and ($(-2,0)+(85:0.4)$) .. (-2,0);
        \fill[violet!10, line width=0.5mm] (-1,0) .. controls ($(-1,0)+(-135:0.95)$) and ($(-2,0)-(85:0.4)$) .. (-2,0);
        
        \draw[violet, line width=0.5mm] (-1,0) .. controls ($(-1,0)+(135:0.95)$) and ($(-2,0)+(85:0.4)$) .. (-2,0);
        \draw[violet, line width=0.5mm] (-1,0) .. controls ($(-1,0)+(-135:0.95)$) and ($(-2,0)-(85:0.4)$) .. (-2,0);

        \fill[orange!10, line width=0.5mm] (-1,0) .. controls ($(-1,0)+(230:0.95)$) and ($(-1,-1)+(180:0.4)$) .. (-1,-1);
        \fill[orange!10, line width=0.5mm] (-1,0) .. controls ($(-1,0)+(-50:0.95)$) and ($(-1,-1)+(0:0.4)$) .. (-1,-1);
        
        \draw[orange, line width=0.5mm] (-1,0) .. controls ($(-1,0)+(230:0.95)$) and ($(-1,-1)+(180:0.4)$) .. (-1,-1);
        \draw[orange, line width=0.5mm] (-1,0) .. controls ($(-1,0)+(-50:0.95)$) and ($(-1,-1)+(0:0.4)$) .. (-1,-1);

        \fill[yellow!80!black!10,line width=0.5mm] (-2.5,0) circle (0.5);
        \fill[teal!10,line width=0.5mm] (-3.5,0) circle (0.5);
        \fill[brown!10,line width=0.5mm] (-2.5,-1) circle (0.5);
        
        \draw[yellow!80!black,line width=0.5mm] (-2.5,0) circle (0.5);
        \draw[teal,line width=0.5mm] (-3.5,0) circle (0.5);
        \draw[brown,line width=0.5mm] (-2.5,-1) circle (0.5);
        \fill (0,0) circle (2pt);
        \fill (-1,0) circle (2pt);
        \fill (-2,0) circle (2pt);
        \fill (-3,0) circle (2pt);
        \fill (-2.5,-0.5) circle (2pt);
        \coordinate[label=center:\small{$1$}] (A1) at (B);
        \coordinate[label=center:\small{$0$}] (A1) at (C);
        \coordinate[label=center:\small{$-1$}] (A1) at (-0.5,0);
        \coordinate[label=center:\small{$-1$}] (A1) at (-1.5,0);
        \coordinate[label=center:\small{$0$}] (A1) at (-2.5,0);
        \coordinate[label=center:\small{$1$}] (A1) at (-3.5,0);
        \coordinate[label=center:\small{$0$}] (A1) at (-2.5,-1);
        \coordinate[label=center:\small{$1$}] (A1) at (-1,-0.5);
        \coordinate[label=center:\small{$\le 1$,}] (A1) at (-4.5,0);
        \end{scope}
    \end{tikzpicture} \begin{tikzpicture}[scale=1.2]
    \begin{scope}[yscale=1,xscale=-1,shift={(-8.5cm,0cm)}]
        \coordinate (O) at (150:0.5);
        \coordinate (P) at (150:0.5);
        \coordinate (B) at (0.25,0.433);
        \coordinate (C) at (0.25,-0.433);
        
        \fill[green!80!black!10,line width=0.5mm] (-1,0) .. controls ($(-1,0)+(45:0.95)$) and ($(0,0)+(100:0.4)$) .. (0,0);
        \fill[green!80!black!10,line width=0.5mm] (-1,0) .. controls ($(-1,0)+(-45:0.95)$) and ($(0,0)+(-100:0.4)$) .. (0,0);
        
        \draw[green!80!black,line width=0.5mm] (-1,0) .. controls ($(-1,0)+(45:0.95)$) and ($(0,0)+(100:0.4)$) .. (0,0);
        \draw[green!80!black,line width=0.5mm] (-1,0) .. controls ($(-1,0)+(-45:0.95)$) and ($(0,0)+(-100:0.4)$) .. (0,0);

        \fill[red!10,line width=0.5mm] (0,0) .. controls (105:0.95) and ($2*(B)+(150:0.4)$) .. ($2*(B)$);
        \fill[red!10,line width=0.5mm] (0,0) .. controls (15:0.95) and ($2*(B)-(150:0.4)$) .. ($2*(B)$);
        
        \draw[red,line width=0.5mm] (0,0) .. controls (105:0.95) and ($2*(B)+(150:0.4)$) .. ($2*(B)$);
        \draw[red,line width=0.5mm] (0,0) .. controls (15:0.95) and ($2*(B)-(150:0.4)$) .. ($2*(B)$);

        \fill[blue!10,line width=0.5mm] (0,0) .. controls (-15:0.95) and ($2*(C)-(210:0.4)$) .. ($2*(C)$);
        \fill[blue!10,line width=0.5mm] (0,0) .. controls (255:0.95) and ($2*(C)+(210:0.4)$) .. ($2*(C)$);
        
        \draw[blue,line width=0.5mm] (0,0) .. controls (-15:0.95) and ($2*(C)-(210:0.4)$) .. ($2*(C)$);
        \draw[blue,line width=0.5mm] (0,0) .. controls (255:0.95) and ($2*(C)+(210:0.4)$) .. ($2*(C)$);
        
        \fill[violet!10, line width=0.5mm] (-1,0) .. controls ($(-1,0)+(135:0.95)$) and ($(-2,0)+(85:0.4)$) .. (-2,0);
        \fill[violet!10, line width=0.5mm] (-1,0) .. controls ($(-1,0)+(-135:0.95)$) and ($(-2,0)-(85:0.4)$) .. (-2,0);
        
        \draw[violet, line width=0.5mm] (-1,0) .. controls ($(-1,0)+(135:0.95)$) and ($(-2,0)+(85:0.4)$) .. (-2,0);
        \draw[violet, line width=0.5mm] (-1,0) .. controls ($(-1,0)+(-135:0.95)$) and ($(-2,0)-(85:0.4)$) .. (-2,0);

        \fill[orange!10, line width=0.5mm] (-1,0) .. controls ($(-1,0)+(230:0.95)$) and ($(-1,-1)+(180:0.4)$) .. (-1,-1);
        \fill[orange!10, line width=0.5mm] (-1,0) .. controls ($(-1,0)+(-50:0.95)$) and ($(-1,-1)+(0:0.4)$) .. (-1,-1);
        
        \draw[orange, line width=0.5mm] (-1,0) .. controls ($(-1,0)+(230:0.95)$) and ($(-1,-1)+(180:0.4)$) .. (-1,-1);
        \draw[orange, line width=0.5mm] (-1,0) .. controls ($(-1,0)+(-50:0.95)$) and ($(-1,-1)+(0:0.4)$) .. (-1,-1);

        \fill[yellow!80!black!10,line width=0.5mm] (-2.5,0) circle (0.5);
        \fill[teal!10,line width=0.5mm] (-3.5,0) circle (0.5);
        \fill[brown!10,line width=0.5mm] (-2.5,-1) circle (0.5);
        
        \draw[yellow!80!black,line width=0.5mm] (-2.5,0) circle (0.5);
        \draw[teal,line width=0.5mm] (-3.5,0) circle (0.5);
        \draw[brown,line width=0.5mm] (-2.5,-1) circle (0.5);
        \fill (0,0) circle (2pt);
        \fill (-1,0) circle (2pt);
        \fill (-2,0) circle (2pt);
        \fill (-3,0) circle (2pt);
        \fill (-2.5,-0.5) circle (2pt);
        \coordinate[label=center:\small{$0$}] (A1) at (B);
        \coordinate[label=center:\small{$1$}] (A1) at (C);
        \coordinate[label=center:\small{$-1$}] (A1) at (-0.5,0);
        \coordinate[label=center:\small{$1$}] (A1) at (-1.5,0);
        \coordinate[label=center:\small{$-2$}] (A1) at (-2.5,0);
        \coordinate[label=center:\small{$1$}] (A1) at (-3.5,0);
        \coordinate[label=center:\small{$1$}] (A1) at (-2.5,-1);
        \coordinate[label=center:\small{$0$}] (A1) at (-1,-0.5);
        \coordinate[label=center:\small{$\le 1$.}] (A1) at (-4.5,0);
        \end{scope}
    \end{tikzpicture}
    \end{center}
\end{example}}

\Xr{Generally, t}\Xn{T}he necessary condition \eqref{condIBI} ensures that inequality \eqref{ineqindcyc} defines a supporting hyperplane of $\CBP(G)$:

\begin{lemma}
    Let $G$ be a connected graph. The independent blocks inequalities define supporting hyperplanes of $\CBP(G)$.
    \label{supporting}
\end{lemma}

\begin{proof}
    Let $\B$ be the set of blocks of $G$. Consider an independent blocks inequality $(\I,\alpha)$ and let $\A \subseteq \B$ be a set of blocks inducing a connected subgraph of $G$. The only positive summands of $\sum_{B \in \B}\alpha_B\chi_B^\A$ are $\alpha_B\chi_B^\A=1$ for $B \in \I \cap \A$. Since $G[\A]$ is connected we have $\chi^\A_B=1$ for $B \in \B\langle \I \cap \A \rangle$ and, hence, it follows
\begin{align*}
    \sum_{B \in \B}\alpha_B \chi^\A_B =\sum_{B \in \A}\alpha_B \le \sum_{B \in \B\langle \I \cap \A\rangle}\alpha_B = \vert \I \cap \A\vert+\sum_{B \in \B\langle \I \cap \A\rangle \setminus \I}\alpha_B \le 1,
\end{align*}
where the last inequality follows by condition \eqref{condIBI} setting $I \coloneqq \I \cap \A$.
\end{proof}

In the following the next observation will be useful that is easy to verify.

\begin{observation}
    \label{sum=1}
    Let $G$ be a connected graph, $\B$ its set of blocks, and $(\I,\alpha)$ an independent blocks inequality. Then
    \begin{align*}
        \sum_{B \in \B}\alpha_B=1.
    \end{align*}
\end{observation}

We will prove in \Cref{indblockfacet} that the independent blocks inequalities define facets. To simplify the proof, we will first describe an inductive construction for these inequalities; see \Cref{fig:Construction} for an illustration. Afterwards, we show that this construction actually yields all possible independent blocks inequalities. We also introduce the following notation. For a connected graph $G$ and a cut vertex $v \in X$ we denote by $\K_v$ the set of components of $G-v$ after adding back $v$ into each component together with its correspondent incident edges\Xn{, see \Cref{fig:K_v} for an example}. Furthermore, for an independent set of blocks $\I=\{B_1,\ldots,B_k\} \subseteq \B$ we denote with $\I_\ell$ the subset $\I_\ell=\{B_1,\ldots,B_\ell\} \subseteq \I$ for $\ell \in [k]$.

\begin{figure}
    \centering
    \begin{tikzpicture}
        \coordinate[label=center:$G$] (A) at (-1,1);
        \fill (0,0) circle (2pt);
        \fill (2,0) circle (2pt);
        \fill (1,-1) circle (2pt);
        \fill (1,1) circle (2pt);
        \fill (4,0) circle (2pt);
        \coordinate[label=above:$v$] (D) at (4,0);
        \fill (3,1) circle (2pt);
        \fill (3,-1) circle (2pt);
        \fill (6.5,1) circle (2pt);
        \fill (6.5,-1) circle (2pt);
        \fill (7.5,0) circle (2pt);
        \fill (9,0) circle (2pt);
        \fill (5.5,-2) circle (2pt);
        \fill (7.5,-2) circle (2pt);
        \fill (1.5,-1) circle (2pt);
        \fill (2,-1) circle (2pt);
        \fill (2,-0.5) circle (2pt);
        \fill (5.5,0) circle (2pt);
        \fill (3,0) circle (2pt);
        \fill (4,-1) circle (2pt);
        
        \draw (0,0) -- (1,1);
        \draw (0,0) -- (2,0);
        \draw (2,0) -- (1,1);
        \draw (2,0) -- (1,-1);
        \draw (2,0) -- (3,-1);
        \draw (1,-1) -- (3,-1);
        \draw (2,0) -- (3,1);
        \draw (2,0) -- (4,0);
        \draw (3,1) -- (4,0);
        \draw (5.5,0) -- (6.5,1);
        \draw (5.5,0) -- (6.5,-1);
        \draw (7.5,0) -- (6.5,1);
        \draw (7.5,0) -- (6.5,-1);
        \draw (6.5,-1) -- (5.5,-2);
        \draw (6.5,-1) -- (7.5,-2);
        \draw (7.5,-2) -- (5.5,-2);
        \draw (7.5,0) -- (9,0);
        \draw (4,0) -- (5.5,0);
        \draw (2,0) -- (1.5,-1);
        \draw (2,0) -- (2,-1);
        \draw (2,-0.5) -- (3,-1);
        \draw (6.5,1) -- (6.5,-1);
        \draw (3,1) -- (3,0);
        \draw (4,0) -- (4,-1);
    \end{tikzpicture}
    
    \bigskip
    
    \bigskip
    
    \begin{tikzpicture}
        \coordinate[label=center:$\K_v$] (A) at (-1,1);
        \fill (0,0) circle (2pt);
        \fill (2,0) circle (2pt);
        \fill (1,-1) circle (2pt);
        \fill (1,1) circle (2pt);
        \fill (4,0) circle (2pt);
        \coordinate[label=right:$v$] (D) at (4,0);
        \fill (3,1) circle (2pt);
        \fill (3,-1) circle (2pt);
        \fill (1.5,-1) circle (2pt);
        \fill (2,-1) circle (2pt);
        \fill (2,-0.5) circle (2pt);
        \fill (3,0) circle (2pt);
        
        \draw (0,0) -- (1,1);
        \draw (0,0) -- (2,0);
        \draw (2,0) -- (1,1);
        \draw (2,0) -- (1,-1);
        \draw (2,0) -- (3,-1);
        \draw (1,-1) -- (3,-1);
        \draw (2,0) -- (3,1);
        \draw (2,0) -- (4,0);
        \draw (3,1) -- (4,0);
        \draw (2,0) -- (1.5,-1);
        \draw (2,0) -- (2,-1);
        \draw (2,-0.5) -- (3,-1);
        \draw (3,1) -- (3,0);

        \fill (5,0) circle (2pt);
        \coordinate[label=above:$v$] (D) at (5,0);
        \fill (5,-1) circle (2pt);
        \draw (5,0) -- (5,-1);

        \fill (6,0) circle (2pt);
        \coordinate[label=left:$v$] (D) at (6,0);
        \fill (8.5,1) circle (2pt);
        \fill (8.5,-1) circle (2pt);
        \fill (9.5,0) circle (2pt);
        \fill (11,0) circle (2pt);
        \fill (7.5,-2) circle (2pt);
        \fill (9.5,-2) circle (2pt);
        \fill (7.5,0) circle (2pt);
        
        \draw (7.5,0) -- (8.5,1);
        \draw (7.5,0) -- (8.5,-1);
        \draw (9.5,0) -- (8.5,1);
        \draw (9.5,0) -- (8.5,-1);
        \draw (8.5,-1) -- (7.5,-2);
        \draw (8.5,-1) -- (9.5,-2);
        \draw (9.5,-2) -- (7.5,-2);
        \draw (9.5,0) -- (11,0);
        \draw (6,0) -- (7.5,0);
        \draw (8.5,1) -- (8.5,-1);
        
    \end{tikzpicture}
    \caption{Example for a graph $G$ and the components of $\K_v$.}
    \label{fig:K_v}
\end{figure}
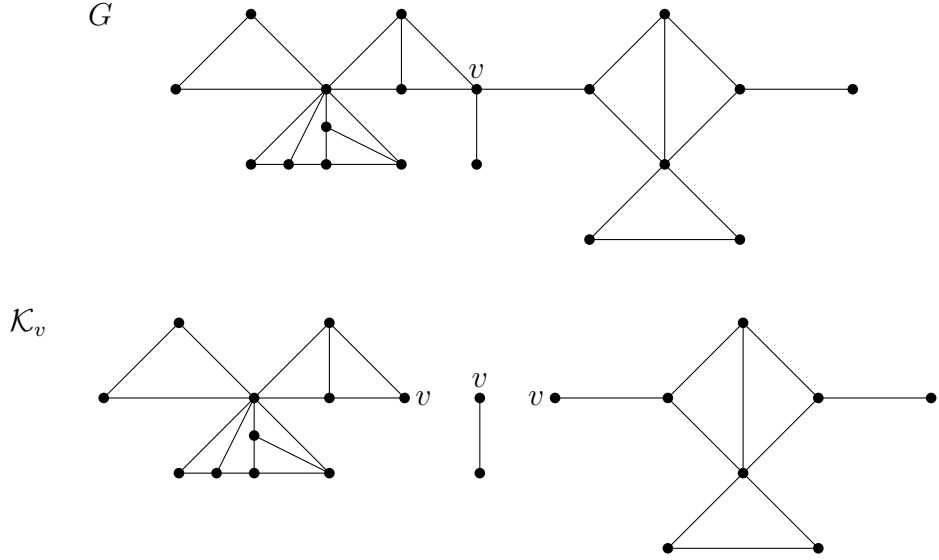

\begin{construction}
\label{construction}
Let $G$ be a connected graph, $\B$ its set of blocks, and $\I=\{B_1,\ldots,B_k\} \subseteq \B$ be an independent set of blocks. We construct pairs $(\I^{(\ell)},\alpha^{(\ell)})$, where $\I^{(\ell)} \subseteq \I_\ell$ and $\alpha^{(\ell)} \in \Z^\B$, for $\ell\in[k]$, such that the pair $(\I^{(\ell)},\alpha^{(\ell)})$ defines an independent blocks inequality. After possibly renumbering the blocks we can assume that $\B \langle\I_\ell \rangle \subsetneq \B \langle \I_{\ell+1} \rangle$ for all $\ell \in [k-1]$.
\begin{description}
    \item[\Xn{Base case}] \, \\ We start by setting $\I^{(1)}=\{B_1\}$, $\alpha^{(1)}_{B_1}=1$, and $\alpha^{(1)}_B=0$ for all $B \in \B \backslash \{B_1\}$. Hence, the first inequality is given by
\begin{align*}
    x_{B_1} \le 1,
\end{align*}
and this defines an independent blocks inequality.
\end{description}

\begin{description}
    \item[\Xn{Construction Step}]\, \\ Assume that $(\I^{(\ell)},\alpha^{(\ell)})$ has already been determined and by induction defines an independent blocks inequality. We obtain $(\I^{(\ell+1)},\alpha^{(\ell+1)})$ as follows. \\
    Consider the unique block path $P \subseteq G$ connecting $B_{\ell+1}$ with $\B\langle\I_\ell\rangle$ not including $B_{\ell+1}$ or any block of $\B\langle\I_\ell\rangle$. \\
    Let $v \in V(G)$ be the cut vertex joining $\B\langle\I_\ell\rangle$ with $P$ (or $B_{\ell+1}$ if $P=\emptyset$). \\
    Let $H \in \K_v$ such that $\sum_{B \in \B[H]} \alpha^{(\ell)}_B=1$.\footnote{\Xn{We will see in \Cref{weight1} that such $H$ exists and is uniquely determined.}}\Xr{We will see in Lemma 5.7 that such $H$ exists and is uniquely determined.} \Xr{Furthermore, we have}\Xn{Observe, that} $B_{\ell+1} \notin \B[H]$. \\
    Let $B'$ be the block of $H$ containing $v$.
    \begin{itemize}
    \item \Xn{Choose $A_\ell \in \B[P] \cup \{B'\}$}.
    \item \Xr{Let $A_\ell \in \B[P] \cup \{B'\}$ and s}\Xn{S}et $\alpha^{(\ell+1)}_{B_{\ell+1}}=1$, $\alpha^{(\ell+1)}_{A_\ell}=\alpha^{(\ell)}_{A_\ell}-1$, and $\alpha^{(\ell+1)}_{B}=\alpha^{(\ell)}_{B}$ for all $B \in \B \setminus \{B_{\ell+1},A_\ell\}$.
    \item If $A_\ell=B'$ and $\alpha^{(\ell)}_{B'}=1$, we set $\I^{(\ell+1)}=(\I^{(\ell)}\setminus \{B'\})\cup \{B_{\ell+1}\}$; otherwise we set $\I^{(\ell+1)}=\I^{(\ell)}\cup \{B_{\ell+1}\}$.
    \end{itemize}
\end{description}
\end{construction}

\begin{figure}
    \centering
    \begin{tikzpicture}
        
        \coordinate (O) at (150:0.5);
        \coordinate (P) at (150:0.5);
        \coordinate (B) at (0.25,0.433);
        \coordinate (C) at (0.25,-0.433);
        
        \draw[orange!20,thick] (0,0) -- (-1,0);
        \fill[orange!20] (0,0) .. controls (120:0.95) and ($(-1,0)+(90:0.4)$) .. (-1,0);
        \fill[orange!20] (0,0) .. controls (-120:0.95) and ($(-1,0)-(90:0.4)$) .. (-1,0);
        \draw (0,0) .. controls (120:0.95) and ($(-1,0)+(90:0.4)$) .. (-1,0);
        \draw (0,0) .. controls (-120:0.95) and ($(-1,0)-(90:0.4)$) .. (-1,0);
        
        \fill[orange!20] (-1.5,0) circle (0.5);
        \fill[orange!20] (-2.5,0) circle (0.5);
        
        \draw (-1.5,0) circle (0.5);
        \draw (-2.5,0) circle (0.5);
        \draw[line width=0.5mm] (-3.5,0) circle (0.5);
        
        \fill (-1,0) circle (2pt);
        \fill (-2,0) circle (2pt);
        \fill (-3,0) circle (2pt);
        
        \coordinate[label=center:$B_5$] (D) at (-3.5,0);
        
        \draw (-2.5,1) circle (0.5);
        \fill (-2.5,0.5) circle (2pt);
        
        \draw ($(-3.5,0)+(120:1)$) circle (0.5);
        \fill ($(-3.5,0)+(120:0.5)$) circle (2pt);
        
        \draw ($(-3.5,0)+(240:1)$) circle (0.5);
        \fill ($(-3.5,0)+(240:0.5)$) circle (2pt);
        
        \draw[cyan!20,thick] (0,0) -- ($2*(B)$);
        \fill[cyan!20] (0,0) .. controls (120:0.95) and ($2*(B)+(150:0.4)$) .. ($2*(B)$);
        \fill[cyan!20] (0,0) .. controls (0:0.95) and ($2*(B)-(150:0.4)$) .. ($2*(B)$);
        \draw (0,0) .. controls (120:0.95) and ($2*(B)+(150:0.4)$) .. ($2*(B)$);
        \draw (0,0) .. controls (0:0.95) and ($2*(B)-(150:0.4)$) .. ($2*(B)$);
        
        \fill[cyan!20] ($3*(B)$) circle (0.5);
        \fill[cyan!20] ($5*(B)$) circle (0.5);
        
        \fill ($2*(B)$) circle (2pt);
        \fill ($4*(B)$) circle (2pt);

        \draw ($3*(B)$) circle (0.5);
        \draw ($5*(B)$) circle (0.5);

        \coordinate[label=center:$B'$] (K) at (B);
        \coordinate[label=center:$A_{2,3}$] (H) at ($5*(B)$);

        \fill[cyan!20] ($5*(B)+(40:1)$) circle (0.5);
        \draw[line width=0.5mm] ($5*(B)+(40:1)$) circle (0.5);
        \fill ($5*(B)+(40:0.5)$) circle (2pt);
        \coordinate[label=center:$B_3$] (G) at ($5*(B)+(40:1)$);
        
        \fill[cyan!20] ($5*(B)+(330:1)$) circle (0.5);
        \draw[line width=0.5mm] ($5*(B)+(330:1)$) circle (0.5);
        \fill ($5*(B)+(330:0.5)$) circle (2pt);
        \coordinate[label=center:$B_4$] (I) at ($5*(B)+(330:1)$);
        
        \fill[cyan!20] ($5*(B)+(110:1)$) circle (0.5);
        \draw[line width=0.5mm] ($5*(B)+(110:1)$) circle (0.5);
        \fill ($5*(B)+(110:0.5)$) circle (2pt);
        \coordinate[label=center:$B_2$] (J) at ($5*(B)+(110:1)$);
        
        \draw ($(I)+(1,0)$) circle (0.5);
        \fill ($(I)+(0.5,0)$) circle (2pt);
        
        \coordinate[label=left:$-2$] (M) at ($5*(B)-(0.5,0)$);
        \coordinate[label=left:$+1$] (N) at ($(J)+(-0.5,0)$);
        \coordinate[label=right:$+1$] (O) at ($(G)+(0.5,0)$);
        \coordinate[label=below:$+1$] (P) at ($(I)+(0,-0.5)$);
        
        \draw[cyan!20,thick] (0,0) -- ($2*(C)$);
        \fill[cyan!20] (0,0) .. controls (0:0.95) and ($2*(C)-(210:0.4)$) .. ($2*(C)$);
        \fill[cyan!20] (0,0) .. controls (240:0.95) and ($2*(C)+(210:0.4)$) .. ($2*(C)$);
        \draw (0,0) .. controls (0:0.95) and ($2*(C)-(210:0.4)$) .. ($2*(C)$);
        \draw (0,0) .. controls (240:0.95) and ($2*(C)+(210:0.4)$) .. ($2*(C)$);
        
        \fill[cyan!20] ($3*(C)$) circle (0.5);
        \draw ($3*(C)$) circle (0.5);
        \fill[cyan!20] ($5*(C)$) circle (0.5);
        \draw[line width=0.5mm] ($5*(C)$) circle (0.5);
        
        \fill ($2*(C)$) circle (2pt);
        \fill ($4*(C)$) circle (2pt);
        \fill ($3*(C)+(210:0.5)$) circle (2pt);
        
        \draw ($3*(C)+(210:1)$) circle (0.5);
        
        \coordinate[label=center:$A_1$] (F) at ($3*(C)$);
        \coordinate[label=center:$B_1$] (E) at ($5*(C)$);
        
        \coordinate[label=right:$-1$] (L) at ($3*(C)+(0.5,0)$);
        \coordinate[label=right:$+1$] (L) at ($5*(C)+(0.5,0)$);
        
        \coordinate[label=left:$v$] (A) at (0,0);
        \fill (A) circle (2pt);
        
        \draw[decorate,decoration={brace,mirror,amplitude=12pt}] (-3,-0.4) -- (-0.23,-0.4) node[midway, below,yshift=-12pt,]{$P$};
        \draw[decorate,decoration={brace,mirror,amplitude=12pt}] (0.45,0) -- ($(I)+(1.65,-0.15)$) node[midway, below,yshift=-12pt,]{$H$};
    \end{tikzpicture}
    \caption{This illustrates an example for a construction step in \Cref{construction}. The blocks are depicted by cycles and the only vertices drawn are the cut vertices. With the notation of \Cref{construction}, we have $\ell=4$, $\I^{(4)}=\{B_1,B_2,B_3,B_4\}$, $\alpha^{(4)}_{B_i}=1$ for $i\in[4]$, $\alpha^{(4)}_{A_1}=-1$, $\alpha^{(4)}_{A_{2,3}}=-2$ and $\alpha^{(4)}_B=0$ for all other blocks $B \in \B$. The blocks of $\I^{(4)}$ and the block $B_5$ have thicker borders, the blocks of $\B\langle\I_4\rangle$ are colored in blue and the blocks of $P$ are colored in orange. The candidates for the block $A_4$ are the three blocks belonging to $P$ and the block $B'$. Then we set $\alpha^{(5)}_{B_5}=1$, $\alpha^{(5)}_{A_4}=-1$ and $\alpha^{(5)}_{B}=\alpha^{(4)}_B$ for all $B \in \B \setminus \{B_5,A_4\}$.}
    \label{fig:Construction}
\end{figure}
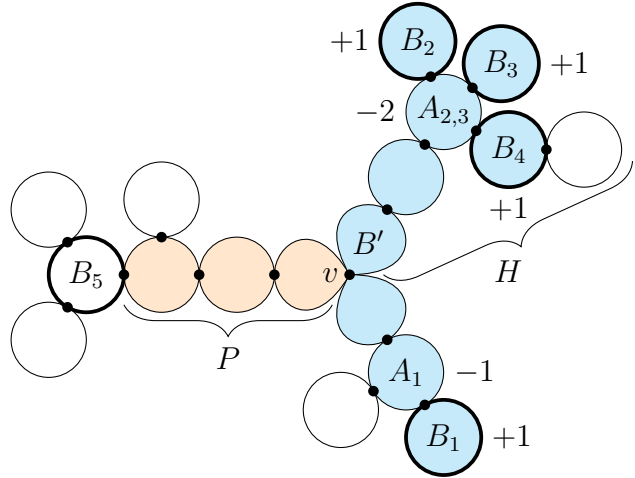

The following lemma shows that \Cref{construction} is well-defined and yields an independent blocks inequality.

\begin{lemma}
\label{weight1}
Let $G$ be a connected graph and $\B$ its set of blocks.
\begin{enumerate}[(i)]
    \item Let $v \in X$ be a cut vertex of $G$ and let $(\I,\alpha)$ be an independent blocks inequality. Then there exists a unique $H \in \K_v$ with
    \begin{align*}
        \sum_{B \in \B[H]}\alpha_B=1,
    \end{align*}
    and we have
    \begin{align*}
        \sum_{B \in \B[H']}\alpha_B=0
    \end{align*}
    for all $H' \in \K_v \setminus \{H\}$.
    \item Let $\I=\{B_1,\ldots,B_k\} \subseteq \B$ be an independent set of blocks. $(\I^{(\ell)},\alpha^{(\ell)})$, as defined by \Cref{construction}, is an independent blocks inequality for all $\ell \in [k]$.
\end{enumerate}
\end{lemma}

\begin{proof}
    \begin{enumerate}[$(i)$]
        \item For any $H'' \in \K_v$ we have $\sum_{B \in \B[H'']}\alpha_B=\sum_{B \in \B}\alpha_B\chi^{\B[H'']} \le 1$ since $(\I,\alpha)$ defines an independent blocks inequality. Suppose there are two subgraphs $H_1 \neq H_2 \in \K_v$ such that $\sum_{B \in \B[H_1]}\alpha_B=\sum_{B \in \B[H_2]}\alpha_B=1$. Then we would have $\sum_{B \in \B[H_1 \cup H_2]}\alpha_B=2$ which is a contradiction since $\B[H_1 \cup H_2]=\B[H_1]\cup\B[H_2]$ induces a connected subgraph of $G$ but $(\I,\alpha)$ defines a feasible independent blocks inequality. If there is no $H \in \K_v$ with $\sum_{B \in \B[H]}\alpha_B=1$, we would have $\sum_{B \in \B}\alpha_B \le 0$ which is again a contradiction due to \Cref{sum=1}. Hence, there is exactly one subgraph $H \in~\K_v$ with $\sum_{B \in \B[H]}\alpha_B=1$. Suppose now there exists an $H' \in K_v\setminus\{H\}$ such that $\sum_{B \in \B[H']}\alpha_B<0$. By \Cref{sum=1} we have
        \begin{align*}
            1=\sum_{B \in \B}\alpha_B=\sum_{B \in \B[H']}\alpha_B+\sum_{H'' \in K_v\setminus \{H'\}}\sum_{B\in\B[H'']}\alpha_B,
        \end{align*}
        and this would imply $\sum_{H'' \in K_v\setminus \{H'\}}\sum_{B\in\B[H'']}\alpha_B > 1$, which is a contradiction since $\bigcup_{H'' \in K_v\setminus \{H'\}}H''$ is a connected subgraph.
        \item This follows by straight-forward induction on $\ell$, where the base case is the fact that $x_{B_1} \le 1$ trivially forms an independent blocks inequality. \qedhere 
    \end{enumerate}
    
\end{proof}

The next lemma shows that \Cref{construction} yields all possible independent blocks inequalities.

\begin{lemma}
\label{lemma}
Let $G$ be a connected graph. Any independent blocks inequality can be obtained by \Cref{construction}.
\end{lemma}

\begin{proof}
We show the statement by induction on the number of blocks in $\B \langle \I \rangle$. If $\vert \B \langle \I \rangle \vert =1$, the independent blocks inequality is simply $x_B \le 1$ for $B \in \B$, which is the base case of \Cref{construction}. Suppose now that $\vert \B \langle \I \rangle \vert >1$. In the following, we deconstruct an independent blocks inequality by undoing a construction step, see \Cref{fig:Proof} for an illustration. Therefore pick a block $B' \in \I$ such that $\B\langle\I\backslash \{B'\}\rangle \subsetneq \B\langle\I\rangle$. Consider the block path $P \subseteq G$ connecting $B'$ with $\B\langle\I\backslash \{B'\}\rangle$ not including $B'$ or any block of $\B\langle\I\backslash \{B'\}\rangle$. Let $v \in V(G)$ be the cut vertex joining $\B\langle\I\backslash \{B'\}\rangle$ with $P$ (or $B'$ if $P=\emptyset$). Let $K \in \K_v$ be the component containing $B'$. By \Cref{weight1} we have $\sum_{B \in \B[K]}\alpha_B \in \{0,1\}$. If $\sum_{B \in \B[K]}\alpha_B =0$, this implies that $\alpha_A=-1$ for a unique $A \in \B[P]$, since $B'$ is the only block $B\in \B[K]$ with $\alpha_B>0$. In this case, set $\I'\coloneqq \I \setminus \{B'\}$, $\beta_{B'} \coloneqq0$,  $\beta_{A}\coloneqq 0$ and $\beta_B\coloneqq \alpha_B$ for all $B \in \B\setminus\{B',A\}$ and we claim that $(\I',\beta)$ is an independent blocks inequality. If $\sum_{B \in \B[K]}\alpha_B =1$, we pick any block $A \in \B$ containing $v$ that does not belong to $K$ and set $\beta_{B'}\coloneqq 0$, $\beta_{A}\coloneqq \alpha_{A}+1$ and $\beta_B\coloneqq \alpha_B$ for all $B \in \B \setminus \{B',A\}$. Furthermore, if $\beta_{A}=1$, we set $\I'=(\I \cup \{A\})\setminus\{B'\}$, and in the other cases, i.e.. $\beta_{A} \le 0$, we set $\I'=\I \setminus \{B'\}$. Again, we claim that $(\I',\beta)$ is an independent blocks inequality, what can be easily verified in both cases.
\end{proof}

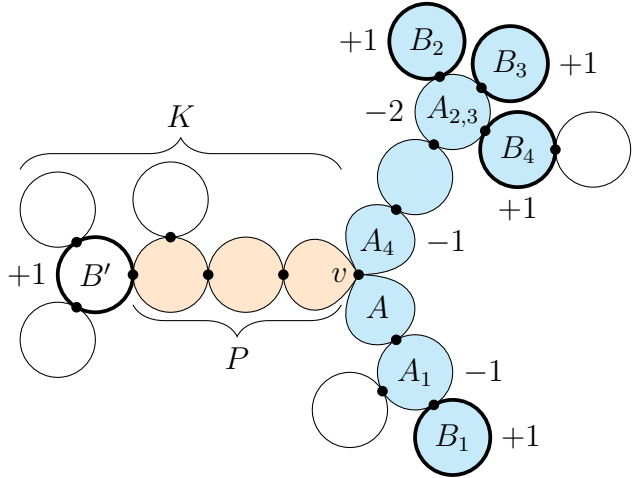
\begin{figure}
    \centering
    \begin{tikzpicture}
        
        \coordinate (O) at (150:0.5);
        \coordinate (P) at (150:0.5);
        \coordinate (B) at (0.25,0.433);
        \coordinate (C) at (0.25,-0.433);
        
        \draw[orange!20,thick] (0,0) -- (-1,0);
        \fill[orange!20] (0,0) .. controls (120:0.95) and ($(-1,0)+(90:0.4)$) .. (-1,0);
        \fill[orange!20] (0,0) .. controls (-120:0.95) and ($(-1,0)-(90:0.4)$) .. (-1,0);
        \draw (0,0) .. controls (120:0.95) and ($(-1,0)+(90:0.4)$) .. (-1,0);
        \draw (0,0) .. controls (-120:0.95) and ($(-1,0)-(90:0.4)$) .. (-1,0);
        
        \fill[orange!20] (-1.5,0) circle (0.5);
        \fill[orange!20] (-2.5,0) circle (0.5);
        
        \draw (-1.5,0) circle (0.5);
        \draw (-2.5,0) circle (0.5);
        \draw[line width=0.5mm] (-3.5,0) circle (0.5);
        
        \fill (-1,0) circle (2pt);
        \fill (-2,0) circle (2pt);
        \fill (-3,0) circle (2pt);
        
        \coordinate[label=center:$B'$] (D) at (-3.5,0);
        \coordinate[label=left:$+1$] (R) at ($(D)+(-0.5,0)$);
        
        \draw (-2.5,1) circle (0.5);
        \fill (-2.5,0.5) circle (2pt);
        
        \draw ($(-3.5,0)+(120:1)$) circle (0.5);
        \fill ($(-3.5,0)+(120:0.5)$) circle (2pt);
        
        \draw ($(-3.5,0)+(240:1)$) circle (0.5);
        \fill ($(-3.5,0)+(240:0.5)$) circle (2pt);
        
        \draw[cyan!20,thick] (0,0) -- ($2*(B)$);
        \fill[cyan!20] (0,0) .. controls (120:0.95) and ($2*(B)+(150:0.4)$) .. ($2*(B)$);
        \fill[cyan!20] (0,0) .. controls (0:0.95) and ($2*(B)-(150:0.4)$) .. ($2*(B)$);
        \draw (0,0) .. controls (120:0.95) and ($2*(B)+(150:0.4)$) .. ($2*(B)$);
        \draw (0,0) .. controls (0:0.95) and ($2*(B)-(150:0.4)$) .. ($2*(B)$);
        
        \fill[cyan!20] ($3*(B)$) circle (0.5);
        \fill[cyan!20] ($5*(B)$) circle (0.5);
        
        \fill ($2*(B)$) circle (2pt);
        \fill ($4*(B)$) circle (2pt);

        \draw ($3*(B)$) circle (0.5);
        \draw ($5*(B)$) circle (0.5);

        \coordinate[label=center:$A_4$] (K) at (B);
        \coordinate[label=center:$A_{2,3}$] (H) at ($5*(B)$);
        
        \coordinate[label=right:$-1$] (S) at ($(K)+(0.5,0)$);

        \fill[cyan!20] ($5*(B)+(40:1)$) circle (0.5);
        \draw[line width=0.5mm] ($5*(B)+(40:1)$) circle (0.5);
        \fill ($5*(B)+(40:0.5)$) circle (2pt);
        \coordinate[label=center:$B_3$] (G) at ($5*(B)+(40:1)$);
        
        \fill[cyan!20] ($5*(B)+(330:1)$) circle (0.5);
        \draw[line width=0.5mm] ($5*(B)+(330:1)$) circle (0.5);
        \fill ($5*(B)+(330:0.5)$) circle (2pt);
        \coordinate[label=center:$B_4$] (I) at ($5*(B)+(330:1)$);
        
        \fill[cyan!20] ($5*(B)+(110:1)$) circle (0.5);
        \draw[line width=0.5mm] ($5*(B)+(110:1)$) circle (0.5);
        \fill ($5*(B)+(110:0.5)$) circle (2pt);
        \coordinate[label=center:$B_2$] (J) at ($5*(B)+(110:1)$);
        
        \draw ($(I)+(1,0)$) circle (0.5);
        \fill ($(I)+(0.5,0)$) circle (2pt);
        
        \coordinate[label=left:$-2$] (M) at ($5*(B)-(0.5,0)$);
        \coordinate[label=left:$+1$] (N) at ($(J)+(-0.5,0)$);
        \coordinate[label=right:$+1$] (O) at ($(G)+(0.5,0)$);
        \coordinate[label=below:$+1$] (P) at ($(I)+(0,-0.5)$);
        
        \draw[cyan!20,thick] (0,0) -- ($2*(C)$);
        \fill[cyan!20] (0,0) .. controls (0:0.95) and ($2*(C)-(210:0.4)$) .. ($2*(C)$);
        \fill[cyan!20] (0,0) .. controls (240:0.95) and ($2*(C)+(210:0.4)$) .. ($2*(C)$);
        \draw (0,0) .. controls (0:0.95) and ($2*(C)-(210:0.4)$) .. ($2*(C)$);
        \draw (0,0) .. controls (240:0.95) and ($2*(C)+(210:0.4)$) .. ($2*(C)$);
        
        \fill[cyan!20] ($3*(C)$) circle (0.5);
        \draw ($3*(C)$) circle (0.5);
        \fill[cyan!20] ($5*(C)$) circle (0.5);
        \draw[line width=0.5mm] ($5*(C)$) circle (0.5);
        
        \fill ($2*(C)$) circle (2pt);
        \fill ($4*(C)$) circle (2pt);
        \fill ($3*(C)+(210:0.5)$) circle (2pt);
        
        \draw ($3*(C)+(210:1)$) circle (0.5);
        
        \coordinate[label=center:$A_1$] (F) at ($3*(C)$);
        \coordinate[label=center:$B_1$] (E) at ($5*(C)$);
        \coordinate[label=center:$A$] (Q) at (C);
        
        \coordinate[label=right:$-1$] (L) at ($3*(C)+(0.5,0)$);
        \coordinate[label=right:$+1$] (L) at ($5*(C)+(0.5,0)$);
        
        \coordinate[label=left:$v$] (A) at (0,0);
        \fill (A) circle (2pt);
        
        \draw[decorate,decoration={brace,mirror,amplitude=12pt}] (-3,-0.4) -- (-0.23,-0.4) node[midway, below,yshift=-12pt,]{$P$};

        \draw[decorate,decoration={brace,amplitude=12pt}] (-4.5,1.4) -- (-0.23,1.4) node[midway, above,yshift=12pt,]{$K$};
    \end{tikzpicture}
    \caption{This illustrates the second case of the proof of \Cref{lemma} in an example. The blocks are depicted by cycles and the only vertices drawn are the cut vertices. In this case we have $\I=\{B_1,B_2,B_3,B_4,B'\}$, $\alpha_{B'}=\alpha_{B_i}=1$ for $i\in[4]$, $\alpha_{A_1}=\alpha_{A_4}=-1$, $\alpha_{A_{2,3}}=-2$ and $\alpha_B=0$ for all other blocks $B \in \B$. The blocks of $\I$ have thicker borders, the blocks of $\B\langle\I\setminus\{B'\}\rangle$ are colored in blue and the blocks of $P$ are colored in orange. We then set $\beta_{B'}=0$, $\beta_{A}=1$, $\beta_B=\alpha_B$ for all $B \in \B \setminus \{B',A\}$ and $\I'=\{B_1,B_2,B_3,B_4,A\}$. It is also possible to set $\beta_{B'}=\beta_{A_4}=0$ and $\beta_B=\alpha_B$ for all $B \in \B\setminus\{B',A_4\}$. In this case $\I'$ is given by $\I'=\{B_1,B_2,B_3,B_4\}$.}
    \label{fig:Proof}
\end{figure}

\begin{theorem}
    \label{indblockfacet}
    Let $G$ be a connected graph. Each independent blocks inequality $(\I,\alpha)$ defines a facet of $\CBP(G)$.
\end{theorem}

\begin{proof}
    The independent blocks inequality $(\I,\alpha)$ are feasible by \Cref{supporting}, and can be obtained by \Cref{construction} by \Cref{lemma}. We show inductively that in each construction step we retain the property of being a facet-defining inequality. In \Cref{theoremBox} we have already shown that $x_B \le 1$ defines a facet for all $B \in \B$.
    
    Generally observe that it suffices to show that $(\I,\alpha')$ -- where $\alpha'\in \Z^{\B\langle\I\rangle}$ with $\alpha'_B=\alpha_B$ for $B \in \B\langle\I\rangle$ -- defines a facet for the connected blocks polytope of the graph $G[\B\langle\I\rangle]$: First of all indeed $(\I,\alpha')$ defines an independent blocks inequality for $G[\B\langle\I\rangle]$ since $\alpha_B=0$ for all $B \notin \B\langle\I\rangle$. Let $\F_{\B\langle\I\rangle}$ be the face defined by $(\I,\alpha')$ of $G[\B\langle\I\rangle]$ and $\F$ be the face defined by $(\I,\alpha)$ of $G$, respectively, and suppose that $\F_{\B\langle\I\rangle}$ is a facet. Let $v$ be a vertex of $\F_{\B\langle\I\rangle}$. Then $w \in \R^\B$, defined by $w_B\coloneqq v_B$ if $B \in \B\langle\I\rangle$ and $w_B\coloneqq 0$ if $B \notin \B\langle \I \rangle$, is a vertex of $\F$ because
    \begin{align*}
        1=\sum_{B \in \B\langle \I\rangle}\alpha_Bv_B=\sum_{B \in \B}\alpha_Bw_B.
    \end{align*}
    Furthermore, let $\B \setminus \B\langle\I\rangle=\{B_1,\ldots,B_k\}$. After possibly reordering, we can assume that $\B^{(i)}\coloneqq \B \cup \{B_1,\ldots,B_i\}$, $i \in [k]$, induce connected subgraphs of $G$. Since $\chi^{\B\langle \I \rangle}$ is a vertex of $\F$ and $\alpha_B=0$ for all $B \in \B \setminus \B \langle \I \rangle$, $\chi^{\B^{(i)}}$ is a vertex of $\F$ for $i \in [k]$. Since these vertices are affinely independet to the before described lifted vertices, we get
    \begin{align*}
        \dim(\F)=\dim(\F_{\B\langle\I\rangle})+k,
    \end{align*}
    which shows that $\F$ is a facet if $\F_{\B\langle\I\rangle}$ is a facet.
    
    Let $(\I^{(\ell)},\alpha^{(\ell)})$ be the independent blocks inequality defined in the $\ell$-th construction step of \Cref{construction} and let $\F^{(\ell)}$ be the corresponding face of $G[\B\langle\I^{(\ell)}\rangle]$. By induction, $\F^{(\ell)}$ is in fact a facet. Let $B_{\ell+1}$ be the new block of $\I^{(\ell+1)}$. There exists a block $B'$ in $\B\langle\I^{(\ell+1)}\rangle$ where the coefficient was decreased by one. Let $P=A_1 \cup \ldots \cup A_m \subseteq G$ be the block path connecting $B_{\ell+1}$ and $\B\langle \I^{(\ell)} \rangle$ not containing $B_{\ell+1}$ or any block of $\B\langle \I^{(\ell)} \rangle$, such that $A_1$ and $B_{\ell+1}$ share a common vertex. We distinguish two cases. Suppose first that $B'=A_{m'}$ for some $m' \in [m]$. Then all vertices of $\CBP(G[\B\langle \I^{(\ell)} \rangle])$ lying in $\F^{(\ell)}$ also belong to $\F^{(\ell+1)}$ by adding zeros at the new coordinates. The vectors
    \begin{align*}
        \chi^{\B\langle \I^{(\ell)} \rangle \cup \{A_j,\ldots,A_m\}}, \, j=m'+1,\ldots,m
    \end{align*}
    and
    \begin{align*}
        \chi^{\{B_{\ell+1}\} \cup \{A_1,\ldots,A_j\}}, \, j=0,\ldots, m'-1
    \end{align*}
    and also the vector $\chi^{\B\langle\I^{(\ell+1)}\rangle}$ are vertices of $\F^{(\ell+1)}$. Hence, the face $\F^{(\ell+1)}$ is a facet of $\CBP(G[\B\langle\I^{(\ell+1)}\rangle])$, since the constructed vertices are affinely independent. For the second case we suppose that the block $B'$ where the coefficient was decreased by one belongs to $\B\langle \I^{(\ell)} \rangle$. Then all vertices $\chi$ of $\F^{(\ell)}$ with $\chi_{B'}=0$ yield vertices of $\F^{(\ell+1)}$ by adding zeros at the new coordinates. For every vertex $\chi'$ of $\F^{(\ell)}$ with $\chi'_{B'}=1$ the vector $\chi'+\chi^{\B[P] \cup \{B_{\ell+1}\}}$ is a vertex of $\F^{(\ell+1)}$. Furthermore, the vectors
    \begin{align*}
        \chi^{\{B_{\ell+1}\} \cup \{A_1,\ldots,A_j\}}, \, j=0,\ldots,m,
    \end{align*}
    belong to $\mathcal{F}$. This shows that $\F^{(\ell+1)}$ is a facet, since again the constructed vertices are affinely independent.
\end{proof}

For a connected graph $G$ and its set of blocks $\B$ we let
\begin{align*}
    \P(G)=\{ x \in [0,1]^{\B} \mid x \text{ satisfies all independent blocks inequalities} \}.
\end{align*}
It is clear that $\CBP(G) \subseteq \P(G)$ and we will show in \Cref{completefacets} that they indeed coincide. For this the following will be useful.

\begin{lemma}
\label{combining}
Let $G$ be a connected graph and $\B$ its set of blocks. Let $(\I^{(1)},\alpha^{(1)}), \ldots, (\I^{(m)},\alpha^{(m)})$ induce independent blocks inequalities for $G$, such that $\B\langle \I^{(i)}\rangle \cap \B\langle \I^{(j)}\rangle= \emptyset$ for all $i,j \in [m]$, $i \neq j$, and $\I \coloneqq \I^{(1)} \cup \ldots \cup \I^{(m)}$ is an independent set of blocks. Moreover, let $\A \subseteq \B$ be a set of blocks, such that $G[\B \setminus \A]$ has exactly $m$ connected components, and each of the connected components contains the blocks $\B\langle \I^{(i)} \rangle$ for exactly one $i \in [m]$. Then there exist blocks $\{A_1,\ldots,A_t\} \subseteq \A$, $t \le m-1$, and weights $\alpha_{A_i} \in \Z_{<0}$, such that $G[\B \setminus \{A_1,\ldots,A_t\}]$ has exactly $m$ connected components $K_1,\ldots,K_m$, with $\B\langle \I^{(i)} \rangle \subseteq K_i$ for all $i \in [m]$, and $(\I,\beta)$, with
\begin{align*}
    \beta_B \coloneqq \begin{cases} \alpha^{(i)}_B, &\text{if } B \in \B\langle \I^{(i)}\rangle, \\
                                    \alpha_{A_i}, &\text{if } B=A_i, \\
                                    0, &\text{otherwise},
    \end{cases}
\end{align*}
defines an independent blocks inequality. Moreover, we have $\sum_{i=1}^t \alpha_{A_i}=-(m-1)$.
\end{lemma}

\begin{proof}
    We show this by induction on $m$, where the statement is clear for $m=1$. Hence, let $m > 1$, and suppose that the above lemma is true for all $m' < m$. By the assumption, there exists an $A \in \B\langle \I \rangle \cap \A $, and we have $ A \notin \B\langle \I^{(i)} \rangle$ for all $i \in [m]$ by the definition of~$\A$. Consider the connected components $K_1,\ldots,K_\ell$ of the graph $G[\B \setminus A]$. By the choice of $A$, we have $\ell \ge 2$, and for all $i \in [m]$, we have $\B\langle \I^{(i)} \rangle \subseteq \B[K_s]$ for some $s \in [\ell]$. In particular, $(\I^{(i)},\alpha^{(i)})$ defines an independent blocks inequality for $\CBP(K_s)$. Moreover, all pairs $(\I^{(i)},\alpha^{(i)})$ belonging to the same connected component $K_s$ together with the set $\A_s \coloneqq \A \cap \B[K_s]$ trivially fulfill the assumptions of this lemma and by the choice of $A$, for each connected component there are at most $m-1$ pairs $(\I^{(i)},\alpha^{(i)})$ belonging to the same connected component. Hence, we can apply the induction hypothesis to $K_s$, and get independent blocks inequalities $(\I'^{(s)},\beta^{(s)})$ for $K_s$ for all $s \in [\ell]$ with the required properties. Then it is easy to check that $(\I,\beta)$, where
    \begin{align*}
    	\beta_B \coloneqq \begin{cases} \beta^{(s)}_B, &\text{if } B \in \B[ K_s], \\
    		-(\ell-1), &\text{if } B=A, \\
    		0, &\text{otherwise},
    	\end{cases}
    \end{align*}
defines an independent blocks inequality with the required properties.
\end{proof}

We can now state the main result of this section.

\begin{theorem}
    \label{completefacets}
    Let $G$ be a connected graph. Then $\CBP(G) = \P(G)$. In particular the independent blocks inequalities, together with nonnegativity constraints, yield the complete and irredundant facet description of $\CBP(G)$.
\end{theorem}

\begin{proof}
    It is clear that $\CBP(G) \subseteq \P(G)$. Let $x \in \P(G)$. We show that there exists a $k \in \mathbb{N}$ and $\lambda_1,\ldots,\lambda_k \in \R_{>0}$, with $\sum_{i=1}^k \lambda_i \le 1$, and vertices $v_1,\ldots,v_k$ of $\CBP(G)$, such that
    \begin{align*}
        x=\sum_{i=1}^k \lambda_iv_i.
    \end{align*}
    This yields that $x$ can be written as a convex combination of the vertices of $\CBP(G)$ since the zero vector can be added to the sum with coefficient $1-\sum_{i=1}^k \lambda_i$. We proceed by induction on the number of blocks in $G$ and show the stronger statement that for every $x \in \P(G)$ there exist $k \in \mathbb{N}$, $\lambda_1,\ldots,\lambda_k \in \R_{>0}$, vertices $v_1,\ldots,v_k$ of $\CBP(G)$, and an independent blocks inequality $(\I,\alpha)$, such that
    \begin{align*}
        x=\sum_{i=1}^k \lambda_iv_i
    \end{align*}
    and
    \begin{align*}
        \sum_{i=1}^k \lambda_i \le \sum_{B \in \B}\alpha_Bx_B.
    \end{align*}
    Since $x \in \P(G)$, this implies $\sum_{i=1}^k \lambda_i \le 1$. Let $\B$ be the set of blocks of $G$. If $\vert \B \vert=1$, the statement is obviously true. Suppose now that $\vert \B \vert=n$ and by induction the statement is true for all connected graphs with less than $n$ blocks. Let $x=(x_{B_1},\ldots,x_{B_n}) \in \P(G)$ and $\mu=\min_{i \in [n]}x_{B_i}$. Then we can write $x$ as
    \begin{align*}
        x=\mu \cdot \mathbf{1}+y,
    \end{align*}
    with $y\coloneqq (x_{B_1}-\mu,\ldots,x_{B_n}-\mu)$. Observe that all entries of $y$ are between $0$ and $1$ and there is at least one entry in $y$ that is $0$. Let $G_y$ be the graph that is given by the union of the blocks of $\supp(y)\coloneqq \{B \in \B \mid y_B > 0\}$ and let $K_1,\ldots,K_m$ be the connected components of $G_y$. For $i\in [m]$ we denote by $y^{(i)}$ the projection of $y$ onto the coordinates $\{B \in \B \mid B \subseteq K_i\}$. We show that $y^{(i)} \in \P(K_i)$ for all $i \in [m]$. Every independent blocks inequality $(\I,\alpha)$ of any $\CBP(K_i)$ induces an independent blocks inequality $(\I,\alpha')$ of $G$, where $\alpha'$ is defined by
    \begin{align*}
        \alpha'_B=\begin{cases}\alpha_B, & \text{if } B \in \B\langle\I\rangle, \\
                                0,& \text{otherwise.}\end{cases}
    \end{align*}
    With this it follows
    \begin{align*}
        \sum_{B \in \B[K_i]}\alpha_By_B^{(i)}&=\sum_{B \in \B}\alpha'_Bx_B-\mu\underbrace{\sum_{B \in \B}\alpha'_B}_{=1} \le 1-\mu\le 1,
    \end{align*}
    which shows that $y^{(i)} \in \P(K_i)$. Hence, by induction, there exist $\lambda_1^{(i)},\ldots,\lambda_{k_i}^{(i)}>0$ and vertices $v_1^{(i)},\ldots,v_{k_i}^{(i)}\in \CBP(K_i)$ and an independent blocks inequality $(\I^{(i)},\alpha^{(i)})$ of $\CBP(K_i)$ such that
    \begin{align*}
        y^{(i)}=\sum_{j=1}^{k_i} \lambda_j^{(i)}v_j^{(i)},
    \end{align*}
    and
    \begin{align}
    \label{lambda}
        \sum_{j=1}^{k_i}\lambda_j^{(i)} \le \sum_{B \in \B[K_i]} \alpha^{(i)}_By^{(i)}_B,
    \end{align}
    Let $w_j^{(i)}\coloneqq v_j^{(i)} \times \{0\}_{B \in \B\setminus\B [K_i]}$. We note that the $w_j^{(i)}$ are vertices of $\CBP(G)$. We can hence write $x$ as
    \begin{align*}
        x=\mu \cdot \mathbf{1} + \sum_{i=1}^m\sum_{j=1}^{k_i} \lambda_j^{(i)}w_j^{(i)}.
    \end{align*}
    For the sum of coordinates we have
    \begin{align*}
        \mu+\sum_{i=1}^m\sum_{j=1}^{k_i} \lambda_j^{(i)} &\overset{\eqref{lambda}}{\le} \mu+\sum_{i=1}^m\sum_{B \in \B[K_i]} \alpha^{(i)}_By_B^{(i)} \\
        &= \mu+\sum_{i=1}^m \left(\sum_{B \in \B[K_i]} \alpha^{(i)}_B(x_B-\mu) \right).
    \end{align*}
    Let $\B(\mu) \coloneqq \{ B \in \B \mid x_B=\mu\}$. Hence, $G[\B \setminus \B(\mu)]=G_y$. By \Cref{combining}, there exists a set of blocks $\A\coloneqq \{A_1,\ldots,A_t\} \subseteq \B(\mu)$, $t \le m-1$, and weights $\alpha_{A_i}$ such that  $(\I,\beta)$, with $\I \coloneqq \bigcup_{i=1}^m\I^{(i)}$, and
    \begin{align*}
        \beta_B\coloneqq \begin{cases} \alpha^{(i)}_B, &\text{if } B \in \B[K_i], \\
                                       \alpha_{A_s}, &\text{if } B=A_s, s \in [t], \\
                                       0, &\text{otherwise}.
        \end{cases}
    \end{align*}
	defines an independent blocks inequality, such that $\sum_{i=1}^t \beta_{A_i}=-(m-1)$. With this we have
    \begin{align*}
        \mu+\sum_{i=1}^m \left(\sum_{B \in \B[K_i]} \alpha^{(i)}_B(x_B-\mu)\right)&=-(m-1)\mu+\sum_{i=1}^m\sum_{B \in \B[K_i]} \alpha^{(i)}_Bx_B \\
        &=\sum_{i=1}^t \beta_{A_i}x_{A_i}+\sum_{i=1}^m\sum_{B \in \B[K_i]} \beta_Bx_B  \\
        &=\sum_{B \in \B}\beta_Bx_B,
    \end{align*}
	which finishes the proof.
\end{proof}

\section{Edges}
\label{sectionEdges}

In this section we fully characterize the edges of the connected blocks polytope and show as an application that connected blocks polytopes are Hirsch. The following theorem considers the polyhedral edges incident to the origin, i.e., the empty subset of blocks.

\begin{theorem}
\label{edgesempty}
    Let $G$ be a connected graph and $\B$ its set of blocks. Consider $\emptyset \neq \A \subseteq \B$ such that $G[\A]$ is connected. Then $\{\textbf{0},\chi^\A\}$ is an edge of $\CBP(G)$ if and only if $\vert \A \vert =1$.
\end{theorem}

\begin{proof}
    
    Due to \Cref{completefacets}, the only homogeneous facets are the faces given by $\{x \in \R^\B \mid x_B=0 \text{ for a } B \in \B\}$. Hence, the vertex figure of the origin is the same as the vertex figure of the origin in the unit cube $Q_{\vert \B \vert} \subseteq \R^\B$.
\end{proof}

The next theorem characterizes the edges of $\CBP(G)$ in the remaining cases.

\begin{theorem}
\label{edges}
    Let $G$ be a connected graph and $\B$ its set of blocks. Let $\emptyset \neq \A_1,\A_2 \subseteq \B$, $\A_1 \neq \A_2$, $G[\A_1]$ and $G[\A_2]$ connected, and $\vert \A_1 \vert \le \vert \A_2 \vert$. Then $\{\chi^{\A_1},\chi^{\A_2}\}$ is an edge of $\CBP(G)$ if and only if one of the following conditions holds:
\begin{enumerate}[(i)]
    \item $G[\A_1 \cup \A_2]$ is disconnected, or
    \item $\A_1 \subseteq \A_2$ and there is a unique block in $\A_2 \setminus \A_1$ that shares a common vertex with a block of $\A_1$.
\end{enumerate}
\end{theorem}

\begin{proof}
    First suppose that $G[\A_1 \cup \A_2]$ is disconnected. Let $\vert \A_1 \vert=k$ and $\vert \A_2 \vert =\ell$. Then
    \begin{align*}
        \H \coloneqq \Big\{ x \in \R^\B \Big\vert \textstyle{\frac{1}{k}}\sum_{B \in \A_1}x_B+\textstyle{\frac{1}{\ell}} \sum_{B \in\A_2}x_B-2\sum_{B \in \B \setminus (\A_1 \cup \A_2)}x_B=1 \Big\} 
    \end{align*}
    is a supporting hyperplane for $\CBP(G)$ and the only vertices contained in this hyperplane are $\chi^{\A_1}$ and $\chi^{\A_2}$. Thus $\{\chi^{\A_1},\chi^{\A_2}\}$ is an edge of $\CBP(G)$.

    Suppose now that $\A_1 \subseteq \A_2$ and that there is a unique block $B' \in \A_2 \backslash \A_1$ that shares a vertex with a block of $\A_1$. Let $m=\vert \A_2 \backslash \A_1 \vert$. We consider the face $\mathcal{F}$ of $\CBP(G)$ defined by
    \begin{align*}
        \F \coloneqq & \{x \in \CBP(G) \mid x_A=1 \text{ for } A \in \A_1, x_A=0 \text{ for } A \notin \A_2\} \\
        =&\conv\{\chi^\A \mid \A_1 \subseteq \A \subseteq \A_2, G[\A] \text{ is connected} \}.
    \end{align*}
    Then the hyperplane
        \begin{align*}
            \H \coloneqq \{x \in \R^\B \mid (m-1) x_{B'}-\textstyle{\sum_{B\in \A_2 \backslash \left(\A_1 \cup \{B'\}\right)}x_B}=0 \}
       \end{align*}
    is a supporting hyperplane for $\F$ and $\chi^{\A_1}$ and $\chi^{\A_2}$ are the only vertices of $\F$ in $\H$ .This shows that $\{\chi^{\A_1},\chi^{\A_2}\}$ is an edge of $\F$ and hence of $\CBP(G)$.

    It remains to show that in all other cases $\{\chi^{\A_1},\chi^{\A_2}\}$ is not an edge of $\CBP(G)$. We first consider the case that $G[\A_1 \cup \A_2]$ is connected but $\A_1 \nsubseteq \A_2$ and $\A_2 \nsubseteq \A_1$. Then $\A_1 \cap \A_2$ and $\A_1 \cup \A_2$ induce connected subgraphs of $G$ different from $\A_1$ and $\A_2$. For their incidence vectors we obtain the relation
    \begin{align*}
        \frac{1}{2}\Big(\chi^{\A_1}+\chi^{\A_2}\Big)=\frac{1}{2}\Big(\chi^{\A_1\cap \A_2}+\chi^{\A_1 \cup \A_2}\Big),
    \end{align*}
    which shows that there is a second pair of vertices of $\CBP(G)$ that has the same midpoint as $\chi^{\A_1}$ and $\chi^{\A_2}$. Hence, $\{\chi^{\A_1},\chi^{\A_2}\}$ is not an edge of $\CBP(G)$.
    
    The second case to consider is that $\A_1 \subseteq \A_2$ and there are at least two blocks $B_1$ and $B_2$ in $\A_2 \backslash \A_1$ that each shares a vertex with some block of $\A_1$ (possibly distinct). Thus, the two graphs $G[\A_1 \cup \{B_1\}]$ and $G[\A_1 \cup \{B_2\}]$ are connected subgraphs of $G$. Now we join the blocks of $\A_2 \backslash (\A_1 \cup \{B_1,B_2\})$ with $\A_1 \cup \{B_1\}$ or $\A_1 \cup \{B_2\}$ iteratively in the following way. We set $\A_1^{(1)}\coloneqq \A_1  \cup \{B_1\}$ and $\A_2^{(1)}\coloneqq \A_1 \cup \{B_2\}$ and $\B^{(1)}\coloneqq \{B_1,B_2\}$. For increasing $\ell$, starting with $\ell=1$, we pick a block $B' \in \A_2 \backslash (\A_1 \cup \B^{(\ell)})$ that shares a common vertex with a block of $\A_1^{(\ell)}$ or $\A_2^{(\ell)}$ (or both). We set $\A_1^{(\ell+1)}\coloneqq \A_1^{(\ell)} \cup \{B'\}$ and $\A_2^{(\ell+1)}\coloneqq \A_2^{(\ell)}$, if $B'$ shares a common vertex with a block of $\A_1^{(\ell)}$, and else, we set $\A_1^{(\ell+1)}\coloneqq \A_1^{(\ell)}$ and $\A_2^{(\ell+1)}\coloneqq \A_2^{(\ell)}\cup \{B'\}$. In both cases we set $\B^{(\ell+1)}\coloneqq \B^{(\ell)}\cup \{B'\}$. By construction it is clear that such a block $B'$ always exists as long as $\A_2 \backslash (\A_1 \cup \B^{(\ell)}) \neq \emptyset$. In the end, we obtain two sets of blocks $\A'_1$ and $\A'_2$ that induce connected subgraphs of $G$ and whose incidence vectors have the same midpoint as $\chi^{\A_1}$ and $\chi^{\A_2}$ in the polytope $\CBP(G)$. Hence, $\{\chi^{\A_1}, \chi^{\A_2}\}$ is not an edge of $\CBP(G)$.
\end{proof}

We recall that a $d$-dimensional polytope $\P$ is Hirsch, if its diameter is at most $k-d$, where $k$ denotes the number of facets of $\P$. In \cite{Dantzig}, Dantzig stated the Hirsch conjecture, first formulated by Hirsch in 1957, which claims that every polytope is Hirsch. More than fifty years later, Santos found the first counterexample to the Hirsch conjecture, see \cite{Santos_2012}. To show that connected blocks polytopes are Hirsch, we first give an upper bound on their diameter.

\begin{lemma}
    \label{diameter}
    Let $G$ be a connected graph. Then $\CBP(G)$ has diameter at most $\dim(\CBP(G))$, i.e., for each pair of vertices there exists a path along the edges of the polytope between those vertices traversing at most $\dim(\CBP(G))$ many edges.
\end{lemma}

\begin{proof}
    Let $\B$ be the set of blocks and let $\delta(\chi_1,\chi_2)$ denote the distance of the vertices $\chi_1$ and $\chi_2$ in the graph of the polytope $\CBP(G)$. At first we remark that if $\A_1,\A_2 \subseteq \B$, such that $G[\A_1]$ and $G[\A_2]$ are connected, with $\A_1 \subseteq \A_2$, we have $\delta(\chi^{\A_1},\chi^{\A_2}) \le \vert \A_2 \setminus \A_1 \vert$: Let $\A_2=\A_1 \cup \{B_1,\ldots,B_k\}$. After possibly reordering we can assume that $\A^{(i)}\coloneqq \A_1 \cup \{B_1,\ldots,B_i\}$ induces a connected subgraph for all $i \in \{0,\ldots,k\}$. Due to \Cref{edgesempty,edges}, respectively, we see that $\{\chi^{\A^{(i-1)}},\chi^{\A^{(i)}}\}$ is an edge of $\CBP(G)$ for all $i \in [k]$. This implies $\delta(\chi^{\A_1},\chi^{\A_2})\le \vert \A_2 \setminus \A_1\vert$.
    
    Now let $\A_1,\A_2 \subseteq \B$ be arbitrary sets of blocks each inducing a connected subgraph. Clearly, $G[\A_1 \cap \A_2]$ also induces a connected subgraph. Then by the previous case, we have
    \begin{align*}
        \delta(\chi^{\A_1},\chi^{\A_2}) &\le \delta(\chi^{\A_1},\chi^{\A_1\cap\A_2})+\delta(\chi^{\A_2},\chi^{\A_1\cap\A_2}) \\
                        & \le \vert \A_1 \setminus (\A_1\cap\A_2) \vert + \vert \A_2 \setminus (\A_1\cap\A_2) \vert \\
                        & \le \vert \B \vert=\dim(\CBP(G)),
    \end{align*}
    where the last inequality follows because the two sets are disjoint.
\end{proof}

From this it easily follows:

\begin{corollary}
    The connected blocks polytopes are Hirsch.
\end{corollary}

\begin{proof}
    Let $\B$ the set of blocks of a connected graph $G$. The dimension of $\CBP(G)$ is $\vert \B \vert$ by \Cref{dimension}. In \Cref{theoremBox}, we showed that the box inequalitites are facets of $\CBP(G)$. Hence, $\CBP(G)$ has at least $2 \cdot \vert \B \vert$ many facets. By \Cref{diameter}, the diameter of $\CBP(G)$ is bounded by
    \begin{align*}
        \dim(\CBP(G))=\vert \B \vert = 2 \cdot \vert \B \vert - \dim(\CBP(G)) \le k-\dim(\CBP(G)),
    \end{align*}
    where $k$ denotes the number of facets of $\CBP(G)$.
\end{proof}

\section{Unimodular Triangulations}

\label{sectionUnimodularTriangulations}

In this section, using our above obtained knowledge of the polytope's facets, we study the $h^\ast$-vector of the connected blocks polytope. We start by providing the needed preliminaries from Ehrhart theory. Afterwards, we show that connected blocks polytopes admit a regular unimodular triangulation by giving a squarefree Gr\"obner basis. Continuing, we show that connected blocks polytopes are Gorenstein of index $2$. Both results together imply unimodality of the $h^\ast$-vector. Furthermore, we will determine the $h^\ast$-polynomial of connected blocks polytopes explicitly, in the case that the underlying graph is a block path. We close this section by giving a lower bound on $h^\ast_1$.

\subsection*{Lattice polytopes and Ehrhart theory}
Let $\Pc \subseteq \R^d$ be a lattice polytope. We say that $\P$ is \emph{reflexive} if its dual polytope $\P^\Delta \subseteq \R^d$ is also a lattice polytope. If $\P$ can be written as
\begin{align*}
    \P=\{x \in \R^d \mid A \cdot x \le \mathbf{1} \},
\end{align*}
where $A$ is an integer matrix, it already follows that $\P$ is reflexive. Furthermore, we say that $\P$ is \emph{Gorenstein of index $k$} if the $k$-th dilation $k\P$ of $\P$ is reflexive.

Ehrhart \cite{Ehrhart} proved that $|n\Pc\cap\Z^d|$ -- i.e.,~the number of lattice points in the $n$-th dilation of $\Pc$ -- is given by a polynomial $E(\Pc,n)$ in $n$ of degree $\dim P$. Similarly, the  \emph{$h^\ast$-polynomial} $h^\ast(\Pc,t)=h_0^\ast(\Pc)+h_1^\ast (\Pc)t+\cdots +h_d^\ast(\Pc) t^d$ of a $d$-dimensional lattice polytope $\Pc$ is obtained by applying a particular change of basis to $E(\Pc,n)$; namely, we have
\[
E(\Pc,n)=h_0^\ast(\Pc) {n+d\choose d}+h_1 ^\ast(\Pc) {n+d-1\choose d} +\cdots + h_d^\ast (\Pc){n\choose d} \, .
\]
All coefficients of $h^\ast(\P,t)$ are nonnegative \cite{Stanley}. For $k\in \N$, $\Pc$ is Gorenstein of index $k$ if $h^\ast(\Pc,t)=t^{d+1-k} h^\ast\left(\Pc,\frac{1}{t}\right)$ \cite{Hib}. 
A \emph{triangulation} $\mathcal{S}$ of a $d$-dimensional lattice polytope $\Pc$ is a subdivision into simplices of dimension at most $d$. 
Such a triangulation is \emph{unimodular} if all its simplices are, i.e., they have normalized volume $1$. The \emph{$f$-polynomial} of $\mathcal{S}$ is defined as 
\[
f(\mathcal{S},t)=f_{-1}+f_0t+\cdots +f_dt^{d+1},
\]
where $f_i=|\{F\in\mathcal{S}\mid \dim(F)=i\}|$. The \emph{$h$-polynomial} $h(\mathcal{S},t)$ of $\mathcal{S}$ is defined via the relation 
\[
f(\mathcal{S},t)=\sum_{i=0}^{d+1}h_it^i(t+1)^{d+1-i}.
\]
If $\Pc$ admits  a unimodular triangulation $\mathcal{S}$, then $h^\ast(\Pc,t)=h(\mathcal{S},t)$. 

To calculate a regular, possibly unimodular, triangulation of $\Pc$, one often computes a Gr\"obner basis. Let $\mathbb{K}$ be a field and $\mathbb{K}[y_1^{\pm},\ldots,y_d^{\pm},s]$ be the Laurent polynomial ring. For a lattice polytope $\Pc\subseteq \R^d$ and any $\alpha=(\alpha_1,\ldots,\alpha_d)\in \Pc\cap \Z^d$, we set $y^\alpha=y_1^{\alpha_1}\cdots y_d^{\alpha_d} \in \mathbb{K}[y_1^{\pm},\ldots,y_d^{\pm},s]$. The ring  $\mathbb{K}[\P]\coloneqq \mathbb{K}[y^\alpha s~|~\alpha\in \Pc\cap \Z^d]$ is the \emph{toric ring} of $\Pc$, i.e., $\mathbb{K}[\P]$ is the subring of $\mathbb{K}[y_1^{\pm},\ldots,y_d^{\pm},s]$ generated by the monomials $y^\alpha s$ for $\alpha\in \Pc\cap \Z^d$. Let $S=\mathbb{K}[x_\alpha~|~\alpha\in \Pc\cap \Z^d]$, where $\deg(x_\alpha)=1$. The \emph{toric ideal} $I_{\Pc}$ of $\Pc$ is the kernel of the surjective ring homomorphism
\[
\pi \colon S\to \mathbb{K}[P] \text{, with } \pi(x_\alpha) \coloneqq y^\alpha s.
\]

A \emph{term order} on a polynomial ring, here $R=\mathbb{K}[x_1,\ldots,x_n]$, is a  total order $<$ on the set of monomials, such that for all monomials $a,b,c$ we have $1<a$, and $a<b$ implies $ac<bc$. A prototypical example is given by the \emph{degree reverse lexicographic ordering} $<_{\rev}$ (degrevlex), induced by $x_1<\cdots < x_n$; we have $x_1^{\alpha_1}\cdots x_n^{\alpha_n}<_{\rev}x_1^{\beta_1}\cdots x_n^{\beta_n}$  if and only if $\sum_{i=1}^n \alpha_i<\sum_{i=1}^n \beta_i$, or  $\sum \alpha_i=\sum \beta_i$ and $\alpha_\ell>\beta_\ell$ for $\ell=\min\{j~|~\alpha_j\neq \beta_j\}$. For a polyomial $f\in R$, its \emph{initial term} $\init_<(f)$ is the largest monomial appearing in $f$. The \emph{initial ideal} $\init_<(I)$ of an ideal $I\subseteq R$ is the ideal generated by all initial terms in $I$. We call a set of  generators $g_1,\ldots,g_r$ of $I$ a \emph{Gr\"obner basis} (w.r.t. $<$) if $\init_<(g_1),\ldots,\init_<(g_r)$ generate $\init_<(I)$. 

Any Gr\"obner basis of the toric ideal $I_{\Pc}$ of a lattice polytope $\Pc$ gives rise to a regular triangulation $\Delta_{\Pc}$ of $\Pc$ on the vertex set $\Pc\cap \Z^d$  as follows: A  set $M\subseteq \Pc\cap \Z^d$ is the vertex set of a simplex in $\Delta_{\Pc}$ if and only if $\prod_{\alpha\in M}x_\alpha\notin \sqrt{\init_<(I_{\Pc})}$. The triangulation $\Delta_{\Pc}$ is a regular unimodular triangulation if and only if $\init_<(I_{\Pc})=\sqrt{\init_<(I_{\Pc})}$. This is the case if and only if $I_{\Pc}$ has a squarefree Gr\"obner basis (see e.g., \cite[Corollary 8.9]{Sturmfels}), i.e., there exists a Gr\"obner basis whose initial terms are squarefree.

\bigskip

\paragraph{\textbf{Results for the connected blocks polytopes}} We now focus on Gröbner bases for connected blocks polytopes. For a connected graph $G$ and its set of blocks $\B$, we consider the polynomial ring $\mathbb{K}[x_\A \mid \A \subseteq \B, \, G[\A] \text{ connected}]$. We consider the partial order where $x_{\A_1} < x_{\A_2}$ if and only if $\A_1 \supset \A_2$. We use the degree reverse lexicographic ordering $<_{\rev}$ w.r.t. any linear extension of this partial order.

\begin{theorem}
    \label{grobner}
    Let $G$ be a connected graph and $\B$ its set of blocks. Then
    \begin{align*}
        \G \coloneqq \{\underline{x_{\A_1}x_{\A_2}}-x_{\A_1 \cap \A_2}x_{\A_1 \cup \A_2} \mid & \A_1,\A_2 \subseteq \B,\ \A_1 \nsubseteq \A_2, \ \A_2 \nsubseteq \A_1, \\
        &G[\A_1],  \ G[\A_2], \ G[\A_1 \cup \A_2] \text{ connected} \}
    \end{align*}
    is a Gr\"obner basis of the toric ideal $I_{\CBP(G)}$ of $\CBP(G)$ with respect to the above defined order. In particular, $\CBP(G)$ has a regular unimodular triangulation. The leading monomials are the underlined terms.
\end{theorem}

\begin{proof}
    Let $\tilde\B =\{\B_1,\ldots,\B_m\}$ be the set of all subsets of blocks inducing a connected subgraph. Let $x^a-x^b \in I_{\CBP(G)}$, $a,b \in \N^{\tilde\B}$, with $\init_{<_\rev}(x^a-x^b)=x^a$. We have to show that there exists $\A_1, \A_2 \in \tilde\B$, $x_{\A_1}x_{\A_2}-x_{\A_1 \cap \A_2}x_{\A_1 \cup \A_2} \in \G$ such that $x_{\A_1}x_{\A_2}$ divides $x^a$, since $\mathrm{in}_<(x_{\A_1}x_{\A_2}-x_{\A_1 \cap \A_2}x_{\A_1 \cup \A_2})=x_{\A_1}x_{\A_2}$. Without loss of generality we can assume that $\gcd(x^a,x^b)=1$. Let $\B[x^a] \coloneqq \bigcup_{a_{\B_i} \neq 0}\B_i$. Since $x^a-x^b \in I_{\CBP(G)}$, we have $\deg(x^a)=\deg(x^b)$ and $\B[x^a]=\B[x^b]$. In particular, $G[\B[x^a]]$ and $G[\B[x^b]]$ have the same connected components. Suppose that there is no element in $\G$ whose initial term divides $x^a$. If there were two sets of blocks $\A'_1, \A'_2 \in \tilde\B$ with $a_{\A'_1},a_{\A'_2} \neq 0$ such that $G[\A'_1 \cup \A'_2]$ is connected, and $\A'_1 \nsubseteq \A'_2$ and $\A'_2 \nsubseteq \A'_1$, then $x_{\A'_1}x_{\A'_2}-x_{\A'_1 \cap \A'_2}x_{\A'_1 \cup \A'_2} \in \G$ and $\init_{<_\rev}(x_{\A'_1}x_{\A'_2}-x_{\A'_1 \cap \A'_2}x_{\A'_1 \cup \A'_2})=x_{\A'_1}x_{\A'_2}$ divides $x^a$. Hence, assume that for all $\A'_1,\A'_2 \in \tilde\B$ with $a_{\A'_1},a_{\A'_2}\neq 0$ and $G[\A_1 \cup \A_2]$ connected, we have $\A'_1 \subseteq \A'_2$ or $\A'_2 \subseteq \A'_1$. Then we have $a_{\B[K]} \neq 0$ for every connected component $K$ of $G[\B[x^a]]=G[\B[x^a]]$ for the following reason. Suppose $a_{\B[K]}=0$ and let $\A \subseteq \B[K]$ be an inclusion-maximal set with $a_\A \neq 0$. By definition of $K$, for every block $B' \in \B[K]$, there has to be an $\A' \in \tilde{\B}$, with $a_{\A'} \neq 0$, that contains $B'$. In particular there is an $\A'$ with $a_{\A'} \neq 0$, that contains a block of $\B[K]\setminus \A$ that shares a common vertex with a block of $\A$. Hence, $G[\A \cup \A']$ is connected and this implies $\A \subsetneq \A'$ which is a contradiction to $\A \subseteq \B[K]$ being inclusion-maximal such that $a_\A \neq 0$. Since $\gcd(x^a,x^b)=1$, we have $b_{\B[K]}=0$ and since $G[\B[x^b]]$ has the same components, for every $\A \in \tilde\B$ with $b_\A \neq 0$ there exists an $\A' \in \tilde\B$ with $\A \subsetneq \A'$ and $a_{\A'} \neq 0$. But because of $\deg(x^a)=\deg(x^b)$ this means that $\mathrm{in}_<(x^a-x^b)=x^b$, which is a contradiction to the assumption that $x^a$ is the initial term of $x^a-x^b$.
    
    Since $\G$ is square free, it follows from \cite[Corollary 8.9]{Sturmfels} that $\CBP(G)$ has a regular unimodular triangulation.
\end{proof}

\begin{theorem}
    \label{gorenstein}
    Let $G$ be a connected graph and $\B$ its set of blocks. Then the connected blocks polytope $\CBP(G)$ is Gorenstein of index $2$.
\end{theorem}

\begin{proof}
    We have to show that the dual polytope of $2 \cdot \CBP(G)- \mathbf{1}$ is a lattice polytope. By \Cref{completefacets} we know that all inequalities defining $\CBP(G)$ are either of the form $x_B \ge 0$ for some $B \in \B$ or independent blocks inequalities. Then the inequalities defining $2 \cdot \CBP(G)$ are either of the form $x_B \ge 0$ or $\sum_{B \in \B}\alpha_Bx_B \le 2$, where there exists an independent set of blocks $\I \subseteq \B$ such that $(\I,\alpha)$ defines an independent blocks inequality. Thus the inequalities describing $2 \cdot \CBP(G)-\mathbf{1}$ are of the form $x_B+1 \ge 0$ and $\sum_{B \in \B}\alpha_B(x_B+1) \le 2$.
    The first inequality is equivalent to $-x_B \le 1$ and the second inequality is equivalent to $\sum_{B \in \B}\alpha_Bx_B \le 1$, since $\sum_{B \in \B} \alpha_B=1$. Since all coefficients are integral, this shows that $2 \cdot \CBP(G)$ is a reflexive polytope by \cite[Theorem 2.11]{Ziegler}.
\end{proof}

It is well known due to Stanley that a result like \Cref{gorenstein} is equivalent to the following statement.

\begin{corollary}
    Let $G$ be a connected graph and $\B$ its set of blocks. The $h^\ast$-vector of $\CBP(G)$ is of the form $(h_0^\ast,h_1^\ast,h_2^\ast,\ldots,h_{n-1}^\ast,h_n^\ast)$, where $n=\vert \B \vert$,
    \begin{align*}
        h_n^\ast&=0, \text{ and} \\
        h^*_i&=h^\ast_{n-i-1} \text{ for } i=0,\ldots,\lfloor \frac{n}{2}\rfloor.
    \end{align*}
\end{corollary}

Furthermore, Theorem 1 of \cite{BR} implies that the $h^\ast$-vector of a Gorenstein lattice polytope with a regular unimodular triangulation is unimodal. Hence, from \Cref{grobner,gorenstein} we get:

\begin{corollary}
    Let $G$ be a connected graph. The $h^\ast$-vector of $\CBP(G)$ is unimodal.
\end{corollary}

\Cref{Narayana} shows the explicit description of the $h^\ast$-vector for the connected blocks polytope of a block path of length $n$, which is given by the $n$-th Narayana polynomial. To this end we define:

\begin{definition}
    For $n \in \N_{>0}$ and $1 \le k \le n$ the Narayana number $N(n,k)$ is defined by
    \begin{align*}
        N(n,k)\coloneqq \frac{1}{n}\binom{n}{k}\binom{n}{k-1}.
    \end{align*}
    The $n$-th Narayana polynomial is given by
    \begin{align*}
        N_n(x) \coloneqq \sum_{k=1}^n N(n,k)x^{k-1}.
    \end{align*}
\end{definition}

\begin{theorem}
    \label{Narayana}
    Let $G'$ be a block path of length $n$. Then the $h^\ast$-polynomial of $\CBP(G')$ is given by the $n$-th Narayana polynomial.
\end{theorem}

\begin{proof}
    Let $\B=\{B_1,\ldots,B_n\}$ be the set of blocks of $G'$, where $B_i$ and $B_{i+1}$ have a common vertex for $i \in [n-1]$. In \cite{Braun15}, Braun showed that the $h^\ast$-polynomial of
    \begin{align*}
        \P_n \coloneqq \conv\{\mathbf{0},\textbf{e}_i-\textbf{e}_j \mid n \ge i > j \ge 1\},
    \end{align*}
    where $\textbf{e}_i$ denotes the $i$-th standard unit vector, is given by the $n-1$-st Narayana polynomial (in \cite{Braun15} it is claimed that $h^\ast(\P_n,z)=N_n(z)$, however, the correct statement is $h^\ast(\P,z)=N_{n-1}(z)$). We will show that $\CBP(G')$ and $\P_{n+1}$ are unimodular equivalent. Our connected blocks polytope is given by
    \begin{align*}
        \CBP(G')=\conv\{\mathbf{0}, \chi^{\{B_k,\ldots,B_\ell\}} \mid 1 \le k \le \ell \le n\}.
    \end{align*}
    $\P_{n+1}$ lies in the $n$-dimensional subspace $V \coloneqq \{x \in \R^{n+1} \mid \sum_{i=1}^n x_i=0\}$. We define the linear map $\varphi \colon \R^\B \to V$ by
    \begin{align*}
        \varphi(\chi^{\{B_i\}})\coloneqq \textbf{e}_{i+1}-\textbf{e}_i,
    \end{align*}
    for all $i \in [n]$. Then $\varphi$ is a unimodular transformation that maps $\CBP(G')$ to $\P_{n+1}$. Thus, $\CBP(G')$ and $\P_{n+1}$ have the same $h^\ast$-polynomial.
    \end{proof}

We finish by giving a lower bound for $h_1^\ast$:

\begin{theorem}
    \label{gamma}
    Let $G$ be a connected graph and $\B$ its set of blocks. Then,
    \begin{align*}
        h_1^\ast \ge \dim \CBP(G)-1
    \end{align*}
    and $h^\ast_1=\dim \CBP(G)-1$ if and only if $\vert \B \vert \le 2$.
\end{theorem}

\begin{proof}
    Let $V(\CBP(G))$ denote the set of vertices of $\CBP(G)$. We have
    \begin{align*}
        h^\ast_1=\vert V(\CBP(G)) \vert -(\dim(\CBP(G))+1).
    \end{align*}
    Hence, we count the vertices of $\CBP(G)$. Let $\B=\{B_1,\ldots,B_n\}$. Without loss of generality we can assume that $G[\{B_1,\ldots,B_i\}]$ is connected for all $i \in [n]$. Then the incidence vectors $\chi^{\{B_1, \ldots, B_i\}}$ yield $\vert \B \vert$ many vertices of $\CBP(G)$. Additionally, there are the vertices $\chi^\emptyset$ and $\chi^B$ for all $B \in \B \backslash \{B_1\}$. Due to \Cref{dimension}, $\vert V(\CBP(G))\vert \ge 2 \cdot \dim(\CBP(G))$ and it follows $h_1^\ast \ge \dim \CBP(G)-1$. For $\vert \B \vert \le 2$ it is clear that these are all the vertices of $\CBP(G)$, which shows $h_1^\ast = \dim \CBP(G)-1$ if $\vert \B \vert \le 2$. For $\vert \B \vert \ge 3$ we know that $B_3$ shares a common node with $B_1$ or $B_2$ by choice of the numbering of $B_1,\ldots,B_n$. Hence, $\chi^{\{B_1 \cup B_3\}}$ or $\chi^{\{B_2 \cup B_3\}}$ is another vertex of $\CBP(G)$. Hence, $h_1^\ast > \dim \CBP(G)-1$.
\end{proof}

For a palindromic $h^\ast$-polynomial the so-called $\gamma$-vector is defined by the coefficients of the $h^\ast$-polynomial after a specific change of basis in the polynomial ring $\R[x]$, and we have $\gamma_1=h^\ast_1-d$, where $d$ denotes the degree of the $h^\ast$-polynomial. For a connected graph $G$, it directly follows from \Cref{gamma} that the $\gamma$-vector of $\CBP(G)$ has the property that $\gamma_1 \ge 0$ and $\gamma_1=0$ if and only if $\vert \B \vert \le 2$, where $\B$ denotes the set of blocks of $G$.

\bibliographystyle{alpha}
\bibliography{bibliography}

Data sharing is not applicable to this article as no new data were created or analyzed in this study.

\end{document}